\documentclass{svjour3}

\usepackage{graphicx}
\usepackage{mathptmx}      
\usepackage{amssymb}
\usepackage{amsfonts}
\usepackage[english]{babel}
\usepackage{amsmath}
\usepackage{enumerate}
\usepackage{epstopdf} 
\usepackage{multirow}
\usepackage{algorithm}
\usepackage{algorithmic}
\usepackage{circuitikz}
\usepackage{subcaption}
\usepackage[utf8]{inputenc}
\usepackage{algorithm}
\usepackage{algorithmic}
\usepackage{tikz}
\usepackage{tikz,graphicx,float}
\usetikzlibrary{shapes.geometric, arrows}
\usetikzlibrary{patterns}
\usetikzlibrary{automata,positioning}
\usetikzlibrary{calc,arrows,shapes,angles,quotes}
\usepackage[margin=2cm]{geometry}

\newtheorem{thm}{Theorem}[section]

\newtheorem{lem}[thm]{Lemma}

\newtheorem{rem}[thm]{Remark}

\numberwithin{equation}{section} \numberwithin{thm}{section}

\newcommand{\be}{\begin{equation}}
\newcommand{\ee}{\end{equation}}
\newcommand{\bea}{\begin{eqnarray}}
\newcommand{\eea}{\end{eqnarray}}

\begin{document}
\sloppy
\title{On Kakeya's Geometric Proof of Enestr\"om-Kakeya's Theorem}



\author{M.~H.~Annaby         \and
        S.~R.~Elsayed-Abdullah
}

\institute{M.~H.~Annaby \at
              Department of Mathematics, Faculty of Science, Cairo University, 12613 Giza, Egypt \\
              \email{mhannaby@sci.cu.edu.eg}           
           \and
           S.~R.~Elsayed-Abdullah \at
           Department of Mathematics, Faculty of Science, Fayoum University, Fayoum, Egypt \\
           \email{sre12@fayoum.edu.eg}
}

\date{Received: date / Accepted: date}
\maketitle

\begin{abstract} 
This paper is devoted to demonstrate Kakeya's geometric proof 
of his theorem (1912), independently established earlier by Enestr\"om (1893). By calculating centers and radii of the interlacing circles of Kakeya's method, we prove Kakeya's geometric structure, which has not been previously established.  We give an equivalent proof, which is based on the construction of internally interlacing circles, which has been geometrically  considered by Tomic (1948).
\end{abstract}



\section{Introduction}

In the following $\Bbb N, \Bbb N_0,\Bbb R,\Bbb C$ denote the sets $\{1,2,\cdots\},\; \Bbb N\cup \{0\},$ the set of real numbers and the set of complex numbers respectively. A complex number is considered as an ordered pair, a point or a vector in $\Bbb R^2.$ Thus we may write for $z=x+iy\in\Bbb C,\; z=\vec{z}=\overrightarrow{OP}=(x,y),\,x,y\in\Bbb R$, where $O$ is the origin and $P$ is the point representing $z$ in the complex plane, or in $\Bbb R^2.$ The circles $\mathcal{C}_1,\;\mathcal{C}_2,\cdots, \mathcal{C}_n,\; n\in \Bbb N$ are called internally interlacing if $\mathcal{C}_k$ is completely contained inside $\mathcal{C}_{k-1},$ except possibly for one touching point, $k=1,2,\cdots,n$. If for all $k$, $\mathcal{C}_{k-1}$ is completely contained inside $\mathcal{C}_{k},$ except possibly for one point, the circles are called externally interlacing. If  $\mathcal{C}_1,\;\mathcal{C}_2,\cdots, \mathcal{C}_n$ are internally interlacing, then  $\mathcal{C}_n,\;\mathcal{C}_{n-1},\cdots, \mathcal{C}_1$ are externally interlacing and vice versa.

Let $a_0,a_1,\cdots,a_n$ be positive numbers, which are not all equal, and $P(z)$ be the polynomial
\begin{equation}\label{eq:1.1}
	P(z)=a_nz^n+a_{n-1}z^{n-1}+\cdots+a_1z+a_0.
\end{equation}
Enestr\"om-Kakeya's theorem states that the zeros of $P(z)$ lie in the annulus, \cite{Enes2,kak},
\begin{equation}\label{eq:1.2}
	\min_{1\leq k \leq n}\frac{a_{k-1}}{a_k} < \left|z\right|<\max_{1\leq k \leq n}\frac{a_{k-1}}{a_k}.
\end{equation}
In particular, if $a_n\geq a_{n-1}\geq \cdots \geq a_1\geq a_0 >0$, then the zeros of $P(z)$ lie in the open unit disc $|z|<1$.

Enestr\"om's algebraic proof,\cite{Enes1,Enes2} is more restrictive than Kakeya's one as he proved that the zeros of $P(z)$ lie in the closure of the  annulus (\ref{eq:1.2}). See also, \cite{Gardner2014} and the brief account \cite[pp. 271-2]{RaScm1}. Kakeya's proof is based on an elegant geometric proof of the following lemma 
\begin{lem}\label{lem:1.1}
If $p_n\geq p_{n-1}\geq \cdots \geq p_1\geq p_0 >0$ are positive numbers, $\theta \in \Bbb R$, $n \in \Bbb N_0$, then
\begin{equation}\label{eq:1.3}
R_n=p_0+p_1e^{i\theta}+p_2e^{2i\theta}+\cdots+p_ne^{in\theta}\ne 0.
\end{equation}
\end{lem}
Kakeya's Proof of Lemma \ref{lem:1.1} is based on proving that the points of $\Bbb C,$
\begin{equation}\label{eq:1.4}
	R_0=p_0,\ R_k=\sum_{m=0}^{k}p_m\,e^{im\theta},\ k=1,2,\cdots,n,
\end{equation}
lie on externally interlacing circles $\mathcal{C}_0,\;\mathcal{C}_1,\cdots, \mathcal{C}_n$ with centers and radii $C_0, C_1,\cdots, C_n$ and $r_0, r_1,\cdots, r_n$ respectively, for which 
\begin{equation}\label{eq:1.5}
	\mathcal{C}_{k-1}\cap\mathcal{C}_k=\{R_{k-1}\}, \ r_k-r_{k-1}=\|C_k-C_{k-1}\|,
\end{equation}
provided that $p_0,p_1,\cdots,p_n$ are all different, $\|\cdot\|$ denotes the Euclidean distance. If some of the $p_j's$ are equal, the geometric proof remains true, but with a slight modification of the geometric structure. If $p_{j-1}=p_j$ for some $j=1,2,\cdots,n,$ then, the circles $\mathcal{C}_{j-1}\mbox{ and }\mathcal{C}_j$ coincide to each other, i.e. $\mathcal{C}_{j-1}=\mathcal{C}_j$ as sets. 
Also,if $p_{j-1}=p_j$, $j=1,2,\cdots,n,$ the lemma is  true only if $e^{i\theta}$ is not an $n$-th root of unity because
\begin{equation}\label{eq:1.6}
	1+z+z^2+\cdots+z^n=\frac{1-z^{n+1}}{1-z},\ z=e^{i\theta}.
\end{equation}

At this point, it is worthwhile to mention that Enestr\"om's algebraic proof and Kakeya's geometric one are identical, apart from the restriction in Enestr\"om's  proof. However, Kakeya's  proof is incomplete as he only  demonstrated  the cases of $\mathcal{C}_0,\,\mathcal{C}_1,\,\mathcal{C}_2$ of  the construction (\ref{eq:1.5}) and claimed that the rest can be derived iteratively without filling in a proof. In addition, he did not provide any calculations of circles centers and radii. Furthermore, his single geometric figure is a troublesome and may be misleading, probably due to the typing-printing limitations of his era.

The first aim of this paper is to derive a complete proof of the geometric construction (\ref{eq:1.5}). We start our demonstrations by computing the particular cases of  $\mathcal{C}_0,\,\mathcal{C}_1,\,\mathcal{C}_2$, which shed light on the general case. Special and general cases are presented in Sections 2 and 3 respectively. We point out to the misleading of Kakeya's own figure. Section 3 involves the general construction (\ref{eq:1.5}) and hence a rigorous geometric proof of (\ref{eq:1.3}) is established. In an equivalent approach, we establish another geometric proof of  Enestr\"om-Kakeya's theorem, which is based on an internally interlacing circles version of Lemma \ref{lem:1.1}. This is done in the last section.

Before leaving this introduction we would like to mention the work of Tomic \cite{Tomic1948}, the first who gave a proof of Kakeya's geometric construction \ref{eq:1.5}. Tomic noticed in \cite{Tomic1948} that the geometric construction of Kakeya's needs a proof and tried to complete it. In his approach he established internally interlacing circles, rather than the externally interlacing ones of Kakeya. Nevertheless, Tomic did not compute the circles, which we rigorously do here.
\section{Special cases}

In the following we prove of (\ref{eq:1.5}) when $n=2$.  We demnostrate the construction of Kakeya for $\mathcal{C}_0,\,\mathcal{C}_1,\,\mathcal{C}_2$. We assume without any loss of generality that $0<\theta<\pi$. The other cases will be discussed in the general setting. The construction of $\mathcal{C}_0$ is based on locating $R_0$ on the $X$-axis at distance $p_0$ from the origin $O$. The point $S_1$ is taken to be $S_1=\overrightarrow{OR_0}+\overrightarrow{R_0S_1}$, $\overrightarrow{R_0S_1}=p_0e^{i\theta}. $Thus $S_1=(p_0+p_0\cos\theta, p_0\sin \theta)$.  The circle $\mathcal{C}_0$ is the circle that passes through $O, R_0, S_1$. From Figure \ref{An1} (a), we see that $\triangle OR_0C_0\cong\triangle R_0S_1C_0$ and consequently $\measuredangle{OC_0R_0}=\measuredangle {R_0C_0O}=\theta$. Therefore, the center and the radius of $\mathcal{C}_0$ are 
\begin{equation}\label{eq:2.1}
\mathcal{C}_0\left(\frac{p_0}{2}, \frac{p_0}{2}\cot \frac{\theta}{2}\right), \ r_0= \frac{p_0}{2}\csc\frac{\theta}{2}.
\end{equation}


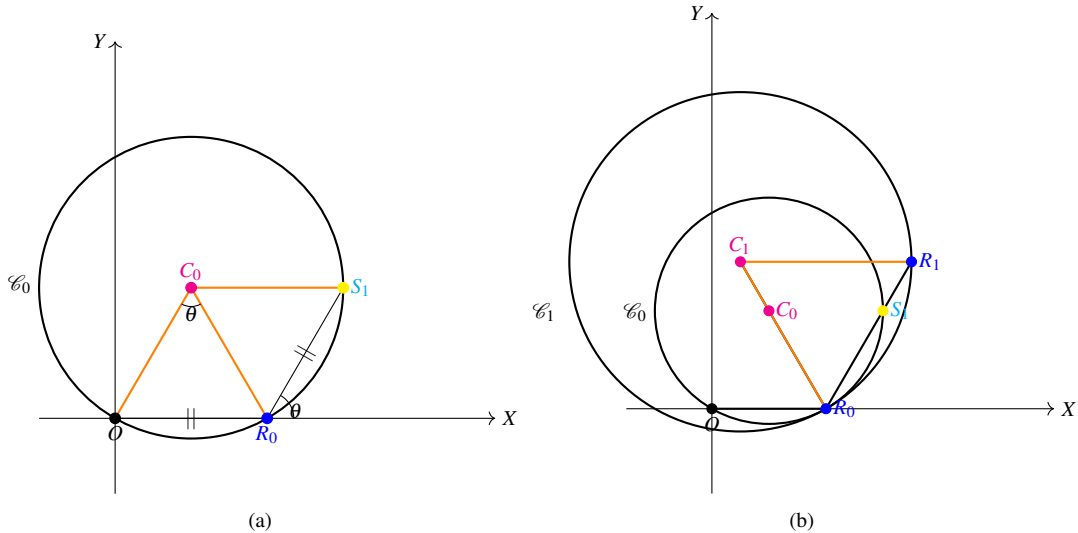
\begin{figure*}[!ht]
	\centering
	\subfloat[]{\begin{tikzpicture}[scale=2]
			\draw[thick] (0.5,0.86602540378)circle(1);
			\draw   (0,0) -- node[sloped] {$||$} (1,0);
			\draw   (1,0) -- node[sloped] {$||$} (1.5,0.86602540378);
			\draw[thick,orange](1,0)--(0.5,0.86602540378);
			\draw[thick,orange](0.5,0.86602540378)--(0,0);
			\draw[thick,orange](1.5,0.86602540378)--(0.5,0.86602540378);
			\draw[->](-0.5,0)--(2.5,0);
			\draw[->](0,-0.5)--(0,2.5);
			\node[below] at (0,0) {$O$};
			\filldraw[black] (0,0) circle (1pt);
			\node[above,magenta] at (0.5,0.86602540378) {$C_0$};
			\filldraw[magenta] (0.5,0.86602540378) circle (1pt);
			\node[below,blue] at (1,0) {$R_0$};
			\filldraw[blue] (1,0) circle (1pt);
			\node[right,cyan] at (1.5,0.86602540378) {$S_1$};
			\filldraw[yellow]  (1.5,0.86602540378) circle (1pt);
			\node[right] at (2.5,0) {$X$};
			\node[left] at (0,2.5) {$Y$};
			\coordinate (a) at (0,0); \coordinate (b) at (0.5,0.86602540378); \coordinate (c) at (1,0); \draw pic[draw,angle radius=0.25cm,"$\theta$",below] {angle=a--b--c};
			\coordinate (a) at (2,0); \coordinate (b) at (1,0); \coordinate (c) at (2,1.73205080757); \draw pic[draw,angle radius=0.35cm,"$\theta$",right] {angle=a--b--c};
			\node[black, below left] at (-0.5,1) {$\mathcal{C}_0$};
	\end{tikzpicture}}
	\subfloat[]{\begin{tikzpicture}[scale=1.5]
			\draw[thick] (0.5,0.86602540378)circle(1);
			\draw[thick] (0.25,1.29903810568)circle(1.5);
			\draw[thick] (0,0)--(1,0); \draw[thick](1,0)--(0.25,1.29903810568);
			\draw[thick,orange](0.25,1.29903810568)--(1,0);
			\draw[thick,orange](0.25,1.29903810568)--(1.75,1.29903810568);
			\draw[thick](1,0)--(1.75,1.29903810568);
			\draw[->](-0.75,0)--(3,0);
			\draw[->](0,-0.75)--(0,3.5);
			\node[below] at (0,0) {$O$};
			\filldraw[black] (0,0) circle (1.25pt);
			\node[right,magenta] at (0.5,0.86602540378) {$C_0$};
			\filldraw[magenta] (0.5,0.86602540378) circle (1.25pt);
			\node[right,blue] at (1,0) {$R_0$};
			\filldraw[blue] (1,0) circle (1.25pt);
			\node[right,cyan] at (1.5,0.86602540378) {$S_1$};
			\filldraw[yellow] (1.5,0.86602540378) circle (1.25pt);
			\node[above,magenta] at (0.25,1.29903810568) {$C_1$};
			\filldraw[magenta] (0.25,1.29903810568) circle (1.25pt);
			\node[right,blue] at (1.75,1.29903810568) {$R_1$};
			\filldraw[blue] (1.75,1.29903810568) circle (1.25pt);
			\node[right] at (3,0) {$X$};
			\node[left] at (0,3.5) {$Y$};
			\node[black, below left] at (-0.5,1) {$\mathcal{C}_0$};
			\node[black, below left] at (-1.3,1) {$\mathcal{C}_1$};
	\end{tikzpicture}}
	\caption{(a) $\triangle OC_0R_0\cong\triangle R_0C_0S_1$, since $OR_0=R_0S_1$. (b) $R_0$, $C_0$ and $C_1$ are collinear as Eq. (\ref{eq:2.4}) confirms.}
	\label{An1}
\end{figure*}


If $p_0=p_1$, then $R_1$ is merely $S_1$, and $\mathcal{C}_1$ coincides with $\mathcal{C}_0$. Otherwise, $R_0S_1$ is extended to $R_0R_1$, where $\overrightarrow{R_0S_1}\parallel \overrightarrow{R_0R_1}$ and $\|R_0R_1\|=p_1$. Thus $R_1=(p_0+p_1\cos \theta, p_1\sin \theta)=\overrightarrow{OR_1}$. The circle $\mathcal{C}_1$ is determined by $R_0, R_1$ and the central angle $\measuredangle R_0C_1R_1$, which is taken to be of measure $\theta$. Therefore $\mathcal{C}_1$ is completely determined and we can see that its center and radius are
\begin{equation}\label{eq:2.2}
\mathcal{C}_1\left(p_0-\frac{p_1}{2}, \frac{p_1}{2}\cot \frac{\theta}{2}\right), \ r_1= \frac{p_1}{2}\csc\frac{\theta}{2}
\end{equation}
respectively. The circle $\mathcal{C}_1$ lies entirely outside $\mathcal{C}_0$ and it only touches it at $R_0$ because 
\begin{equation}\label{eq:2.3}
r_1-r_0=\|C_1-C_0\|=\frac{p_1-p_0}{2}\csc\frac{\theta}{2}.
\end{equation}

It is worthwhile to notice that $R_0,\, C_0,\, C_1$ are collinear, cf. Figure \ref{An1} (b). This is clear by (\ref{eq:2.3}) and 
\begin{eqnarray}\label{eq:2.4}
\begin{vmatrix}
p_0 & 0  & 1 \\
\frac{p_0}{2} & \frac{p_0}{2}\cot \frac{\theta}{2} & 1 \\
p_0-\frac{p_1}{2} & \frac{p_1}{2}\cot \frac{\theta}{2} & 1 
\end{vmatrix}=0.
\end{eqnarray}


\noindent In Kakeya's single figure of \cite{kak}, $R_0,\,C_0,\; C_1$ are not looking to be collinear.

The construction of $\mathcal{C}_2$ is similar to that of $\mathcal{C}_1$. We first locate $S_2$ whose position vector is 
\begin{equation}\label{eq:2.5}
\overrightarrow{OS_2}=\overrightarrow{OR_0}+\overrightarrow{R_0R_1} +\overrightarrow{R_1S_2},\quad  \overrightarrow{R_1S_2}= p_1e^{2i\theta}.
\end{equation}
Thus the coordinates of $S_2$ are 
\begin{equation}\label{eq:2.6}
S_2\left(p_0+p_1\cos \theta+p_1\cos 2\theta, p_1\sin \theta+p_1\sin 2\theta \right).
\end{equation}
It is a trigonometry exercise to check that $S_2$ lies on $\mathcal{C}_1$. Again if $p_1=p_2$, then $S_2=R_2$, and $\mathcal{C}_2$ coincides with $\mathcal{C}_1$. Otherwise we extend $\overrightarrow{R_1S_2}$ to $\overrightarrow{R_1R_2}$,
\begin{equation}\label{eq:2.7}
\overrightarrow{R_1R_2}\parallel \overrightarrow{R_1S_2}, \quad \|\overrightarrow{R_1R_2}\|=p_2.
\end{equation}
The circle $\mathcal{C}_2$ is determined by $R_1$ and $R_2$ and the central angle $\measuredangle R_1C_2R_2$, which is taken to satisfy $\measuredangle {R_1C_1R_2}=\theta$. 
We can see that the center and radius of $\mathcal{C}_2$ are 
\begin{equation}\label{eq:2.8}
\mathcal{C}_2\left(p_0+p_1\cos \theta+\frac{p_2}{2}\cos 2\theta- \frac{p_2}{2}\sin 2\theta \cot \frac{\theta}{2}, p_1\sin \theta+\frac{p_2}{2}\sin 2\theta+ \frac{p_2}{2}\cos 2\theta \cot \frac{\theta}{2}\right), \ r_2= \frac{p_2}{2}\csc\frac{\theta}{2}.
\end{equation}
Simple calculations prove that $R_1, C_1, C_2$ are collinear, and
\begin{equation}\label{eq:2.9}
r_2-r_1=\|C_2-C_1\|=\frac{p_2-p_1}{2}\csc\frac{\theta}{2},
\end{equation}
i.e. $\mathcal{C}_2$ lies entirely outside $\mathcal{C}_1$ and it touches it only at $R_2$. The general case will be proved below.

Figures \ref{fig3}-\ref{fig5} illustrate the constructions of  $\mathcal{C}_0,\,\mathcal{C}_1,\,\mathcal{C}_2$ for different values of $\theta$. Figure \ref{fig3} illustrates the construction of $\mathcal{C}_0,\,\mathcal{C}_1,\,\mathcal{C}_2$ when $\theta=\frac\pi3$ is an acute angle. We considered two cases when $p_0,\,p_1,\,p_2$ are all different and when two values coincide to each other. The cases when $\theta$ is right or obtuse angles are depicted in Figure \ref{fig4} and Figure \ref{fig5} respectively. The calculations coincide with the geometric construction.  The cases when $\theta=0,\,\pi$ or when $\theta$ lies in the third or fourth quadrants are discussed in the general setting of the next section. 

\begin{figure*}[!ht]
	\centering
	\subfloat[]{\begin{tikzpicture}
			\draw[thick] (0.5,0.86602540378)circle(1);
			\draw[thick] (0,1.73205080757)circle(2);
			\draw[thick] (-1,1.73205080757)circle(3);
			\draw[thick] (0,0)--(1,0); \draw[thick,blue](1,0)--(0,1.73205080757);
			\draw[thick](0.5,0.86602540378)--(0,0);
			\draw[thick](2,1.73205080757)--(1,0);
			\draw[thick](2,1.73205080757)--(0,1.73205080757);
			\draw[thick](2,1.73205080757)--(0.5,4.33012701892);
			\draw[thick](-1,1.73205080757)--(0.5,4.33012701892);
			\draw[thick,magenta](-1,1.73205080757)--(2,1.73205080757);
			\node[below] at (0,0) {$O$};\filldraw[black] (0,0) circle (2pt);
			\node[right,magenta] at (0.5,0.86602540378) {$C_0$}; \filldraw[magenta] (0.5,0.86602540378) circle (2pt); \node[below,blue] at (1,0) {$R_0$}; \filldraw[blue] (1,0) circle (2pt);
			\node[right,cyan] at (1.5,0.86602540378) {$S_1$};  \filldraw[yellow] (1.5,0.86602540378) circle (2pt);
			\node[above,magenta] at (0,1.73205080757) {$C_1$};
			\filldraw[magenta] (0,1.73205080757) circle (2pt);
			\node[right,blue] at (2,1.73205080757) {$R_1$};
			\filldraw[blue] (2,1.73205080757) circle (2pt);
			\node[left,cyan] at (1.4,3.6) {$S_2$};
			\filldraw[yellow]  (1.02,3.46) circle (2pt);
			\node[left,magenta] at (-1,1.73205080757) {$C_2$}; 
			\filldraw[magenta] (-1,1.73205080757) circle (2pt);
			\node[above,blue] at (0.5,4.33012701892) {$R_2$};
			\filldraw[blue] (0.5,4.33012701892) circle (2pt); 
			\draw[thick,dashed] (1,0)--(2.6,0);
			\draw[thick,dashed] (2,1.73205080757)--(3,3.46410161514);
			\coordinate (a) at (0,0); \coordinate (b) at (0.5,0.86602540378); \coordinate (c) at (1,0); \draw pic[draw,angle radius=0.25cm,"$\theta$",below] {angle=a--b--c};
			\coordinate (a) at (1,0); \coordinate (b) at (0,1.73205080757); \coordinate (c) at (2,1.73205080757); \draw pic[draw,angle radius=0.25cm,"$\theta$",below right] {angle=a--b--c};
			\coordinate (a) at (2,1.73205080757); \coordinate (b) at (-1,1.73205080757); \coordinate (c) at (0.5,4.33012701892); \draw pic[draw,angle radius=0.25cm,"$\theta$", above right] {angle=a--b--c};
			\coordinate (a) at (3,3.46410161514); \coordinate (b) at (2,1.73205080757); \coordinate (c) at (0.5,4.33012701892); \draw pic[draw,angle radius=0.35cm,"$\theta$",above right] {angle=a--b--c};
			\coordinate (a) at (2,0); \coordinate (b) at (1,0); \coordinate (c) at (2,1.73205080757); \draw pic[draw,angle radius=0.25cm,"$\theta$",right] {angle=a--b--c};
			\node[black, below left] at (-0.5,1) {$\mathcal{C}_0$};
			\node[black, below left] at (-2,1) {$\mathcal{C}_1$};
			\node[black, below left] at (-4,1) {$\mathcal{C}_2$};
	\end{tikzpicture}}
	\subfloat[]{\begin{tikzpicture}[scale=1.45]
			\draw[thick] (0.5,0.86602540378)circle(1);
			\draw[thick] (0,1.73205080757)circle(2);
			\draw[thick] (0,0)--(1,0); \draw[thick,blue](1,0)--(0,1.73205080757);
			\draw[thick](0.5,0.86602540378)--(0,0);
			\draw[thick](2,1.73205080757)--(1,0);
			\draw[thick,magenta](2,1.73205080757)--(0,1.73205080757);
			\draw[thick](2,1.73205080757)--(1.05,3.45);
			\draw[thick](1.05,3.45)--(0,1.73205080757);
			\node[below,red] at (0,0) {$O$};
			\filldraw[black] (0,0) circle (1.5pt);\node[right,magenta] at (0.5,0.86602540378) {$C_0$};
			\filldraw[magenta] (0.5,0.86602540378) circle (1.5pt); \node[below,blue] at (1,0) {$R_0$};
			\filldraw[blue] (1,0) circle (1.5pt);
			\node[right,cyan] at (1.5,0.86602540378) {$S_1$};
			\filldraw[yellow]  (1.5,0.86602540378) circle (1.5pt);
			\node[left,magenta] at (0,1.73205080757) {$C_1=C_2$};
			\filldraw[magenta]  (0,1.73205080757) circle (1.5pt);
			\node[right,blue] at (2,1.73205080757) {$R_1$};
			\filldraw[blue]  (2,1.73205080757) circle (1.5pt);
			\node[right,blue] at (1.05,3.45) {$S_2=R_2$};
			\filldraw[blue]  (1.05,3.45) circle (1.5pt);
			\draw[thick,dashed] (1,0)--(2.6,0);
			\draw[thick,dashed] (2,1.73205080757)--(3,3.46410161514);
			\coordinate (a) at (0,0); \coordinate (b) at (0.5,0.86602540378); \coordinate (c) at (1,0); \draw pic[draw,angle radius=0.25cm,"$\theta$",below] {angle=a--b--c};
			\coordinate (a) at (1,0); \coordinate (b) at (0,1.73205080757); \coordinate (c) at (2,1.73205080757); \draw pic[draw,angle radius=0.25cm,"$\theta$",below right] {angle=a--b--c};
			\coordinate (a) at (2,1.73205080757); \coordinate (b) at (0,1.73205080757); \coordinate (c) at (1.05,3.45); \draw pic[draw,angle radius=0.3cm,"$\theta$", above right] {angle=a--b--c};
			\coordinate (a) at (3,3.46410161514); \coordinate (b) at (2,1.73205080757); \coordinate (c) at (0.5,4.33012701892); \draw pic[draw,angle radius=0.35cm,"$\theta$",above right] {angle=a--b--c};
			\coordinate (a) at (2,0); \coordinate (b) at (1,0); \coordinate (c) at (2,1.73205080757); \draw pic[draw,angle radius=0.35cm,"$\theta$",right] {angle=a--b--c};
			\node[black, below left] at (-0.5,1) {$\mathcal{C}_0$};
			\node[black, below left] at (-2,1) {$\mathcal{C}_1=\mathcal{C}_2$};
	\end{tikzpicture}}
	
	\caption{The construction of $\mathcal{C}_0,\,\mathcal{C}_1,\,\mathcal{C}_2$ when $\theta=\frac{\pi}{3}$. (a) $p_0=1, p_1=2, p_2=3$ are different. (b) $p_0=1, p_1=p_2=2$ and $\mathcal{C}_1,\,\mathcal{C}_2$ are identical. Accidentally  $C_1$ lies on $\mathcal{C}_0.$ Notice that $C_2$ does not lie  on $\mathcal{C}_1$.}
	\label{fig3}
\end{figure*}
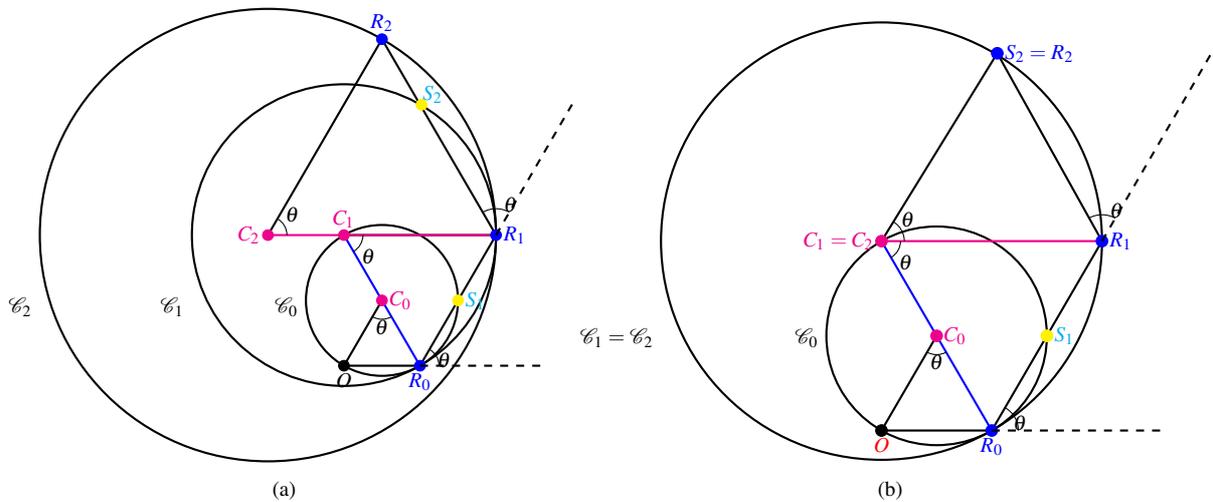


\begin{figure*}[!ht]
	\centering
	\subfloat[]{\begin{tikzpicture}[scale=1.5]
		\draw[thick] (0.5,0.5)circle(0.70710678118);
		\draw[thick] (0,1)circle(1.41421356237);
		\draw[thick] (-0.5,0.5)circle(2.12132034356);
		\draw[thick] (0,0)--(1,0); \draw[thick,blue](1,0)--(0,1);
		\draw[thick](0.5,0.5)--(0,0);
		\draw[thick](1,2)--(1,0);
		\draw[thick](1,2)--(-2,2);
		\draw[thick,magenta](1,2)--(-0.5,0.5);
		\draw[thick](-0.5,0.5)--(-2,2);
		\node[below] at (0,0) {$O$};
		\filldraw[black]  (0,0) circle (1.5pt);
		\node[right,magenta] at (0.5,0.5) {$C_0$}; 
		\filldraw[magenta]  (0.5,0.5) circle (1.5pt);
		\node[below,blue] at (1,0) {$R_0$};
		\filldraw[blue]  (1,0) circle (1.5pt);
		\node[right,cyan] at (1,1) {$S_1$};
		\filldraw[yellow]  (1,1) circle (1.5pt);
		\node[above,magenta] at (0,1) {$C_1$};
		\filldraw[magenta]  (0,1) circle (1.5pt);
		\node[right,blue] at (1,2) {$R_1$};
		\filldraw[blue]  (1,2) circle (1.5pt);
		\node[left,blue] at (-2,2) {$R_2$};
		\filldraw[blue]  (-2,2) circle (1.5pt);
		\node[left,magenta] at (-0.5,0.5) {$C_2$}; 
		\filldraw[magenta]  (-0.5,0.5) circle (1.5pt);
		\node[above,cyan] at (-1,2) {$S_2$}; 
		\filldraw[yellow]  (-1,2) circle (1.5pt);
		\draw[thick,dashed] (1,0)--(2,0);
		\draw[thick,dashed] (1,2)--(1,3);
		\coordinate (a) at (0,0); \coordinate (b) at (0.5,0.5); \coordinate (c) at (1,0); \draw pic[draw,angle radius=0.25cm,"$\theta$",below] {angle=a--b--c};
		\coordinate (a) at (1,0); \coordinate (b) at (0,1); \coordinate (c) at (1,2); \draw pic[draw,angle radius=0.25cm,"$\theta$", right] {angle=a--b--c};
		\coordinate (a) at (1,2); \coordinate (b) at (-0.5,0.5); \coordinate (c) at (-2,2); \draw pic[draw,angle radius=0.25cm,"$\theta$", above] {angle=a--b--c};
		\coordinate (a) at (1,3.6); \coordinate (b) at (1,2); \coordinate (c) at (-1,2); \draw pic[draw,angle radius=0.35cm,"$\theta$",above] {angle=a--b--c};
		\coordinate (a) at (2,0); \coordinate (b) at (1,0); \coordinate (c) at (1,2); \draw pic[draw,angle radius=0.35cm,"$\theta$",right] {angle=a--b--c};
		\node[black, below left] at (-0.15,0.25) {$\mathcal{C}_0$};
		\node[black, below left] at (-1.25,0.25) {$\mathcal{C}_1$};
		\node[black, below left] at (-2.75,0.25) {$\mathcal{C}_2$};
	\end{tikzpicture}}
	\subfloat[]{\begin{tikzpicture}[scale=2.2]
			\draw[thick] (0.5,0.5)circle(0.70710678118);
			\draw[thick] (0,1)circle(1.41421356237);
			\draw[thick] (0,0)--(1,0); \draw[thick,blue](1,0)--(0,1);
			\draw[thick](0.5,0.5)--(0,0);
			\draw[thick](1,2)--(1,0);
			\draw[thick](1,2)--(-1,2);
			\draw[thick, magenta](1,2)--(0,1);
			\draw[thick](-1,2)--(0,1);
			\node[below] at (0,0) {$O$};
			\filldraw[black]  (0,0) circle (1.0pt);
			\node[right, magenta] at (0.5,0.5) {$C_0$};
			\filldraw[magenta]  (0.5,0.5) circle (1.0pt); \node[below,blue] at (1,0) {$R_0$};
			\filldraw[blue]  (1,0) circle (1.0pt);
			\node[right,green] at (1,1) {$S_1$};
			\filldraw[yellow]  (1,1) circle (1.0pt);
			\node[left,magenta] at (0,1) {$C_1=C_2$};
			\filldraw[magenta]  (0,1) circle (1.0pt);
			\node[right,blue] at (1,2) {$R_1$};
			\filldraw[blue]  (1,2) circle (1.0pt);
			\node[above left,blue] at (-1,2) {$S_2=R_2$}; 
			\filldraw[blue]  (-1,2) circle (1.0pt);
			\draw[thick,dashed] (1,0)--(2,0);
			\draw[thick,dashed] (1,2)--(1,3);
			\coordinate (a) at (0,0); \coordinate (b) at (0.5,0.5); \coordinate (c) at (1,0); \draw pic[draw,angle radius=0.15cm,"$\theta$",below] {angle=a--b--c};
			\coordinate (a) at (1,0); \coordinate (b) at (0,1); \coordinate (c) at (1,2); \draw pic[draw,angle radius=0.15cm,"$\theta$", right] {angle=a--b--c};
			\coordinate (a) at (1,2); \coordinate (b) at (0,1); \coordinate (c) at (-1,2); \draw pic[draw,angle radius=0.25cm,"$\theta$", above] {angle=a--b--c};
			\coordinate (a) at (1,3.6); \coordinate (b) at (1,2); \coordinate (c) at (-1,2); \draw pic[draw,angle radius=0.35cm,"$\theta$",above] {angle=a--b--c};
			\coordinate (a) at (2,0); \coordinate (b) at (1,0); \coordinate (c) at (1,2); \draw pic[draw,angle radius=0.35cm,"$\theta$",right] {angle=a--b--c};
			\node[black, below left] at (-0.15,0.25) {$\mathcal{C}_0$};
			\node[black, below left] at (-1.25,0.25) {$\mathcal{C}_1=\mathcal{C}_2$};
	\end{tikzpicture}}
	
\caption{The circles $\mathcal{C}_0,\,\mathcal{C}_1,\,\mathcal{C}_2$ with a right angle  $\theta=\frac{\pi}{2}$. (a) $p_0=1, p_1=2, p_2=3$ are different. (b) $p_0=1, p_1=p_2=2$ and $\mathcal{C}_1,\,\mathcal{C}_2$ are identical.}
\label{fig4} 
\end{figure*}
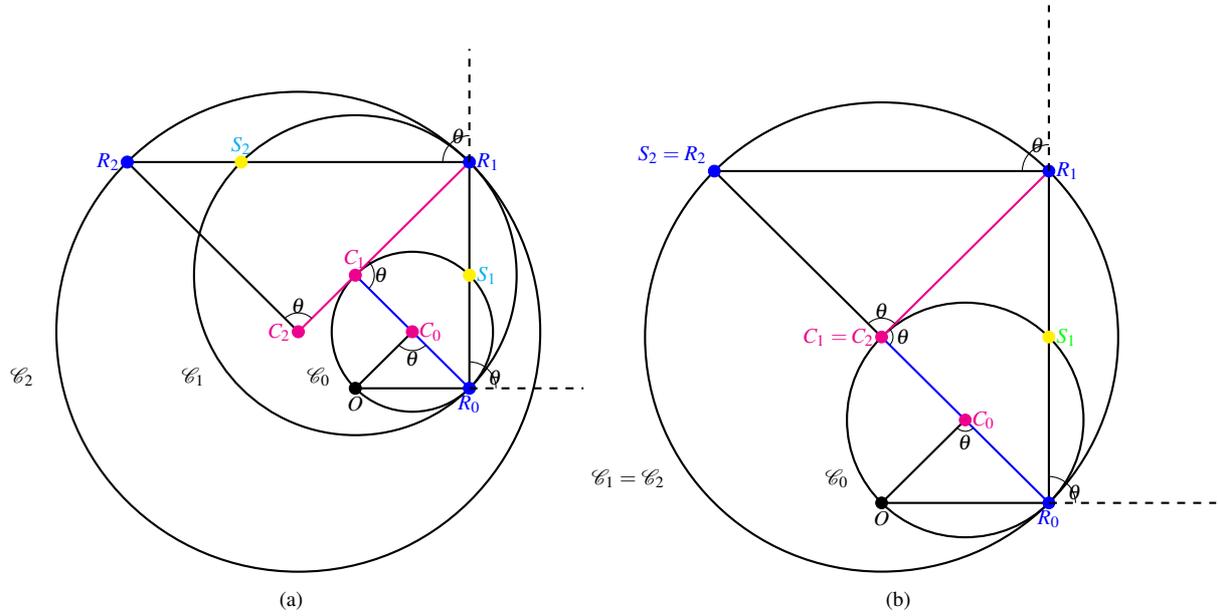


\begin{figure*}[!ht]
	\centering
	\subfloat[]{\begin{tikzpicture}[scale=1.7]
			\draw[thick] (0.5,0.28867513459)circle(0.57735026919);
			\draw[thick] (0,0.57735026919)circle(1.15470053838);
			\draw[thick] (0,0)circle(1.73205080757);
			\draw[thick] (0,0)--(1,0); \draw[thick,blue](1,0)--(0,0.57735026919);
			\draw[thick](0.5,0.28867513459)--(0,0);
			\draw[thick](0, 1.73205080757)--(1,0);
			\draw[thick,magenta](0, 1.73205080757)--(0,0);
			\draw[thick](0, 1.73205080757)--(-1.5,-0.86602540378);
			\draw[thick](0,0)--(-1.5,-0.86602540378);
			\node[below,magenta] at (0,0) {$C_2$};
			\filldraw[magenta]  (0,0) circle (1.25pt);
			\node[above,magenta] at (0.5,0.28867513459) {$C_0$};
			\filldraw[magenta]  (0.5,0.28867513459) circle (1.25pt);
			 \node[right,below,blue] at (1,0) {$R_0$};
			 \filldraw[blue]  (1,0) circle (1.25pt);
			\node[left,blue] at (-1.5,-0.86602540378) {$R_2$};
			\filldraw[blue]  (-1.5,-0.86602540378) circle (1.25pt);
			\node[left,magenta] at (0,0.57735026919) {$C_1$};
			\filldraw[magenta]  (0,0.57735026919) circle (1.25pt);
			\node[above,cyan] at (0.5,0.86602540378) {$S_1$};
			\filldraw[yellow] (0.5,0.86602540378) circle (1.25pt);
			\node[left,cyan] at (-1,0) {$S_2$};
			\filldraw[yellow] (-1,0) circle (1.25pt);
			\node[above,blue] at (0,1.73205080757) {$R_1$}; 
			\filldraw[blue] (0,1.73205080757) circle (1.25pt);
			\draw[thick,dashed] (1,0)--(2,0);
			\draw[thick,dashed] (-0.5,2.59807621135)--(0,1.73205080757);
			\coordinate (a) at (0,0); \coordinate (b) at (0.5,0.28867513459); \coordinate (c) at (1,0); \draw pic[draw,angle radius=0.15cm,"$\theta$",below] {angle=a--b--c};
			\coordinate (a) at (1,0); \coordinate (b) at (0,0.57735026919); \coordinate (c) at (0,1.73205080757); \draw pic[draw,angle radius=0.15cm,"$\theta$",right] {angle=a--b--c};
			\coordinate (a) at (0,1.73205080757); \coordinate (b) at (0,0); \coordinate (c) at (-1.5,-0.86602540378); \draw pic[draw,angle radius=0.15cm,"$\theta$", above left] {angle=a--b--c};
			\coordinate (a) at (-0.5,2.59807621135); \coordinate (b) at (0,1.73205080757); \coordinate (c) at (-1.5,-0.86602540378); \draw pic[draw,angle radius=0.25cm,"$\theta$",above left] {angle=a--b--c};
			\coordinate (a) at (2,0); \coordinate (b) at (1,0); \coordinate (c) at (0,1.73205080757); \draw pic[draw,angle radius=0.25cm,"$\theta$",right] {angle=a--b--c};
			\node[black, below left] at (1,1.2) {$\mathcal{C}_0$};
			\node[black, below left] at (-1.1,1) {$\mathcal{C}_1$};
			\node[black, below left] at (-2,1) {$\mathcal{C}_2$};
	\end{tikzpicture}}
	\subfloat[]{\begin{tikzpicture}[scale=2.5]
			\draw[thick] (0.5,0.28867513459)circle(0.57735026919);
			\draw[thick] (0,0.57735026919)circle(1.15470053838);
			\draw[thick] (0,0)--(1,0); \draw[thick](1,0)--(0.5,0.28867513459);
			\draw[thick](0.5,0.28867513459)--(0,0);
			\draw[thick, blue](0,0.57735026919)--(1,0);
			\draw[thick](0, 1.73205080757)--(0,0.57735026919);
			\draw[thick](0, 1.73205080757)--(-1,0);
			\draw[thick](0,0.57735026919)--(-1,0);
			\draw[thick](1,0)--(0,1.73205080757);
			\draw[thick, magenta](0,0)--(0,1.73205080757);
			\node[below] at (0,0) {$O$};
			\filldraw[black] (0,0) circle (0.8pt);
			\node[above,magenta] at (0.5,0.28867513459) {$C_0$};
			\filldraw[magenta] (0.5,0.28867513459) circle (0.8pt);
			 \node[right,below,blue] at (1,0) {$R_0$};
			 \filldraw[blue] (1,0) circle (0.8pt);
			\node[below,magenta] at (0,0.57735026919) {$C_1$};
			\filldraw[magenta] (0,0.57735026919) circle (0.8pt);
			\node[above,cyan] at (0.5,0.86602540378) {$S_1$};
			\filldraw[yellow](0.5,0.86602540378) circle (0.8pt);
			\node[left,blue] at (-1,0) {$S_2=R_2$};
			\filldraw[blue](-1,0) circle (0.8pt);
			\node[above,blue] at (0,1.73205080757) {$R_1$}; 
			\filldraw[blue](0,1.73205080757) circle (0.8pt);
			\draw[thick,dashed] (1,0)--(2,0);
			\draw[thick,dashed] (-0.5,2.59807621135)--(0,1.73205080757);
			\coordinate (a) at (0,0); \coordinate (b) at (0.5,0.28867513459); \coordinate (c) at (1,0); \draw pic[draw,angle radius=0.15cm,"$\theta$",below] {angle=a--b--c};
			\coordinate (a) at (1,0); \coordinate (b) at (0,0.57735026919); \coordinate (c) at (0,1.73205080757); \draw pic[draw,angle radius=0.15cm,"$\theta$",right] {angle=a--b--c};
			\coordinate (a) at (0,1.73205080757); \coordinate (b) at (0,0.57735026919); \coordinate (c) at (-1,0); \draw pic[draw,angle radius=0.25cm,"$\theta$", above left] {angle=a--b--c};
			\coordinate (a) at (-0.5,2.59807621135); \coordinate (b) at (0,1.73205080757); \coordinate (c) at (-1.5,-0.86602540378); \draw pic[draw,angle radius=0.25cm,"$\theta$",above left] {angle=a--b--c};
			\coordinate (a) at (2,0); \coordinate (b) at (1,0); \coordinate (c) at (0,1.73205080757); \draw pic[draw,angle radius=0.25cm,"$\theta$",right] {angle=a--b--c};
			\node[black, below left] at (0,-0.25) {$\mathcal{C}_0$};
			\node[black, below left] at (-1.1,1) {$\mathcal{C}_1=\mathcal{C}_2$};
	\end{tikzpicture}}
	
\caption{The construction of $\mathcal{C}_0,\,\mathcal{C}_1,\,\mathcal{C}_2$ when $\theta=\frac{2\pi}{3}$. (a) $p_0=1, p_1=2, p_2=3$ are different. Accidentally $C_2$ coincides with the origin. (b) $p_0=1, p_1=p_2=2$ and $\mathcal{C}_1,\,\mathcal{C}_2$ are identical.}
\label{fig5} 
\end{figure*}
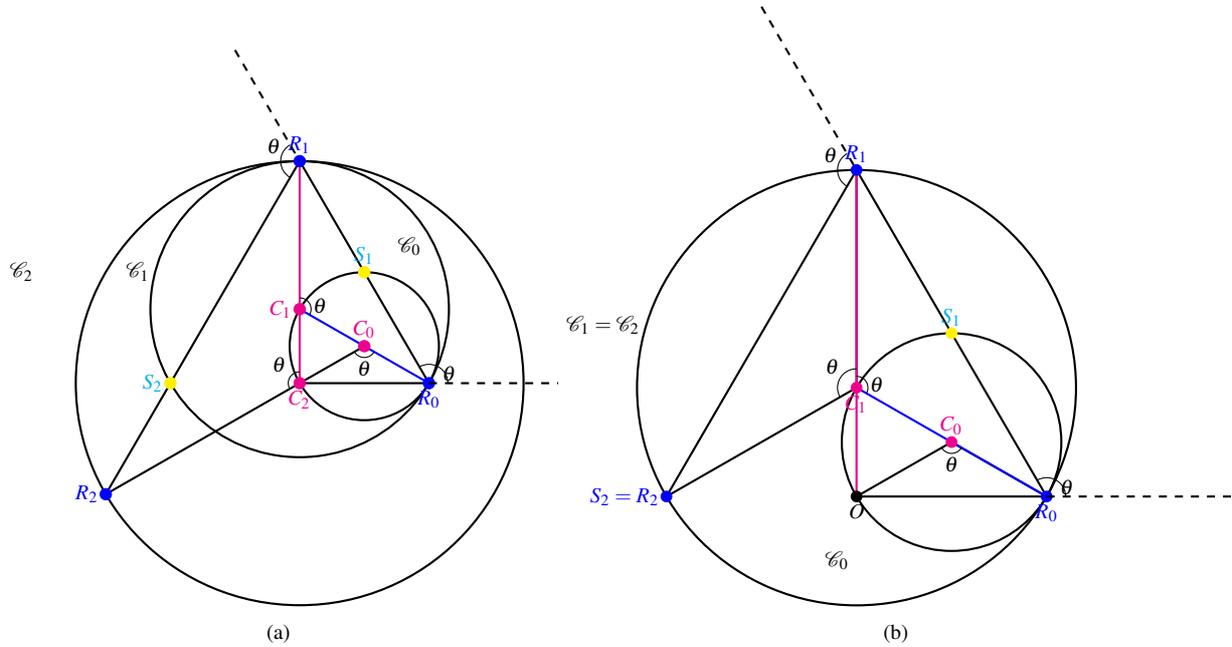


\section{The general setting}
In this section we provide a complete proof of the geometric structure (\ref{eq:1.5}) which implies a rigorous geometric proof of (\ref{eq:1.3}), and consequently the proof of Enestr\"om-Kakeya's theorem.

\noindent{\bf Proof of Lemma \ref{lem:1.1}.} We assume without any loss of generality that $0<\theta<\pi$ and that $p_0,\,p_1,\,\cdots,\,p_n$ are distinct. Let $R_k$ have the coordinates $(x_k,y_k),\,k=0,1,\cdots,n$. Hence
\begin{eqnarray}\label{A1}
x_k&=&\left\{\begin{array}{ll}
p_0+p_1\cos\theta+\cdots+p_{k}\cos k\theta,& k\ge1,\\
{}&{}\\
p_0,&k=0,
\end{array}
\right.\\
\label{A2}
y_k&=&\left\{\begin{array}{ll}
p_1\sin\theta+\cdots+p_{k}\sin k\theta,& k\ge1,\\
		{}&{}\\
		0,&k=0.
	\end{array}
	\right.
\end{eqnarray}
The points $S_k,\,k=1,\cdots,n$ are defined via
\begin{equation}\label{A3}
\overrightarrow{OS_{k}}=\left\{\begin{array}{ll}
		\overrightarrow{OR_0}+\overrightarrow{R_0R_1} +\cdots+\overrightarrow{R_{k-2}R_{k-1}}+p_{k-1}e^{ik\theta},& k\ge2,\\
		{}&{}\\
		\overrightarrow{OR_0}+p_0e^{i\theta},&k=1.
	\end{array}
	\right.
\end{equation}
Thus, if the coordinates of $S_k,\,k=1,\cdots,n$ are $(s_{k,x}, s_{k,y})$, then 
\begin{eqnarray}\label{A4}
s_{k,x}&=&\left\{\begin{array}{ll}
	p_0+p_1\cos\theta+\cdots+p_{k-1}\cos (k-1)\theta+p_{k-1}\cos k\theta,& k\ge2,\\
	{}&{}\\
	p_0+p_0\cos\theta,&k=1,
\end{array}
\right.\\
\label{A5}
s_{k,y}&=&\left\{\begin{array}{ll}
	p_1\sin\theta+\cdots+p_{k-1}\sin (k-1)\theta+p_{k-1}\sin k\theta,& k\ge2,\\
	{}&{}\\
	p_0\sin\theta,&k=1,
\end{array}
\right.
\end{eqnarray}

We proved (\ref{eq:1.5}) for $\mathcal{C}_0,\,\mathcal{C}_1,\,\mathcal{C}_2$. Now we accomplish the prove for all $k.$ Indeed, let $k\ge2$. By Kakeya structure, the circle $\mathcal{C}_k$ is constructed to pass through $R_{k-1},\,R_k$ and $\measuredangle {R_{k-1}C_kR_k
}=\theta$. Thus $\mathcal{C}_k$ is completely determined. From triangle $\triangle R_{k-1}C_kR_k$, cf. Figure \ref{An2} (a) the radius of $\mathcal{C}_k$ is
\begin{equation}\label{eq:3.6}
r_k=\frac{p_k}{2}\csc \frac{\theta}{2}.
\end{equation}
To find the center $C_k(h_k,l_k)$, we notice that $C_kM_k\perp R_{k-1}R_k,$ where $M_k$ is the mid-point of $R_{k-1}R_k$, cf. Figure \ref{An2} (a) . Hence, $h_k,\;l_k$ are determined from solving the system
\begin{eqnarray}\label{A6}
	&l_k=\frac{y_{k-1}+y_k}{2}-\frac{x_{k}-x_{k-1}}{y_{k}-y_{k-1}}\left(h_k-\frac{x_{k-1}+x_k}{2}\right),\\
	\label{A7}
&\left(x_k-h_k\right)^2+\left(y_k-l_k\right)^2=\frac{p_k^2}{4}\,\csc^2\frac\theta2,\\
	\label{A8}
&\left(x_{k-1}-h_k\right)^2+\left(y_{k-1}-l_k\right)^2=\frac{p_k^2}{4}\,\csc^2\frac\theta2.
\end{eqnarray}
By algebraic manipulations, the center $C_k(h_k,l_k),\,k\ge2$ is determined via
\begin{eqnarray}\label{A9}
h_k&=&
\frac{x_{k-1}+x_k}{2}-\frac{y_k-y_{k-1}}{2}\,\cot \frac{\theta}{2},\\
\label{A10}
l_k&=& \frac{y_{k-1}+y_k}{2}+\frac{x_k-x_{k-1}}{2}\,\cot \frac{\theta}{2}.
\end{eqnarray}
Using (\ref{A1})-(\ref{A2}) , we obtain
\begin{eqnarray}\label{A11}
	h_k&=&
	p_0+p_1\cos\theta+\cdots+p_{k-1}\cos(k-1)\theta+\frac{p_k}{2}\cos k\theta-\frac{p_k}{2}\sin k\theta \cot \frac{\theta}{2},\\
	\label{A12}
	l_k&=& p_1\sin\theta+\cdots+p_{k-1}\sin(k-1)\theta+\frac{p_k}{2}\sin k\theta+\frac{p_k}{2}\cos k\theta \cot \frac{\theta}{2}.
\end{eqnarray}
Hence, taking into account the results of the previous section, all circles $\mathcal{C}_0,\,\mathcal{C}_1,\cdots,\mathcal{C}_n$ are completely determined. Now we prove that they satisfy  (\ref{eq:1.5}). Indeed, by construction $R_{k-1}\in\mathcal{C}_{k-1}\cap\mathcal{C}_k.$ This can be also verified by direct computations in the equations of $\mathcal{C}_{k-1},\;\mathcal{C}_k.$ Furthermore, $\mathcal{C}_{k-1}\cap\mathcal{C}_k=\left\{ R_{k-1}\right\}$ and $\mathcal{C}_k$ lies entirely outside $\mathcal{C}_{k-1}$, except for they touch each other at $R_{k-1}$ because 
\begin{eqnarray}\label{A13}
	\nonumber\parallel C_{k+1}-C_k\parallel^2&=&\left(\left(\frac{p_k}{2}\cos k \theta+\frac{p_{k+1}}{2}\cos(k+1)\theta\right)+\left(\frac{p_k}{2}\sin k \theta \cot \frac{\theta}{2}-\frac{p_{k+1}}{2}\sin(k+1)\theta \cot \frac{\theta}{2}\right)\right)^2\\
	\nonumber{}{}&{}&+\left(\left(\frac{p_k}{2}\sin k \theta+\frac{p_{k+1}}{2}\sin(k+1)\theta\right)+\left(-\frac{p_k}{2}\cos k \theta \cot \frac{\theta}{2}+\frac{p_{k+1}}{2}\cos(k+1)\theta \cot \frac{\theta}{2}\right)\right)^2\\
	\nonumber{}{}&=&\frac{p_k^2}{2}+\frac{p_{k+1}^2}{2}+\frac{p_kp_{k+1}}{2}\cos \theta +\cot^2\frac{\theta}{2}\left(\frac{p_k^2}{2}-\frac{p_{k+1}^2}{2}+\frac{p_kp_{k+1}}{2}\cos \theta \right)-p_kp_{k+1}\cot \frac{\theta}{2}\sin \theta\\
	{}{}&=&\left(\frac{p_k^2}{4}-\frac{p_kp_{k+1}}{2}+\frac{p_{k+1}^2}{4}\right)\csc^2\frac{\theta}{2}=\left(\frac{p_{k+1}-p_k}{2}\csc \frac{\theta}{2}\right)^2=\left(r_{k+1}-r_k\right)^2.
\end{eqnarray}
Hence (\ref{eq:1.5}) is proved, completing the proof of Lemma \ref{lem:1.1}.\qed


\begin{figure*}[!ht]
	\centering
	\subfloat[]{\begin{tikzpicture}[rotate=25, scale=1.25]
			\draw[thick] (1,1.5)circle(1.80277563773);
			\draw[thick] (0.5,2.25)circle(2.7041634566);
			\draw[thick] (0,0)--(2,0); \draw[thick](2,0)--(1,1.5);
			\draw[thick](1,1.5)--(0,0);
			\draw[thick](0.5,2.25)--(3.15,2.76);
			\draw[thick](2,0)--(3.15,2.76);
			\node[below,blue] at (0,0) {$R_{k-1}$};
			\filldraw[blue] (0,0) circle (2.0pt);
			\node[below,violet] at (1,0) {$M_{k}$};
			\filldraw[violet] (1,0) circle (2.0pt);
			\node[below,blue] at (0.5,0.5) {$\frac{p_{k}}{2}$};
			\filldraw[blue] (0,0) circle (2.0pt);
			\node[below,blue] at (1.5,0.5) {$\frac{p_{k}}{2}$};
			\node[above,right,magenta] at (1,1.5) {$C_k$};
			\filldraw[magenta] (1,1.5) circle (2.0pt);
			 \node[above,magenta] at (0.5,2.25) {$C_{k+1}$};
			\filldraw[magenta] (0.5,2.25) circle (2.0pt);
			 \node[right,below,blue] at (2,0) {$R_k$};
			\filldraw[blue] (2,0) circle (2.0pt);
			\node[right,below,blue] at (3.55,2.76) {$R_{k+1}$};
			\filldraw[blue] (3.15,2.76) circle (2.0pt);
			\node[right,cyan] at (36/13,24/13) {$S_{k+1}$};
			\filldraw[yellow] (36/13,24/13) circle (2.0pt);
			\draw[thick,dashed] (1,1.5)--(1,0);
			\draw[thick,dashed] (1,1.5)--(0.5,2.25);
			\draw[thick,dashed] (2,0)--(3,0);
			\draw[thick,dashed] (3.15,2.76)--(3.6,3.84);
			\coordinate (a) at (0,0); \coordinate (b) at (1,1.5); \coordinate (c) at (2,0); \draw pic[draw,angle radius=0.3cm,"$\theta$",below] {angle=a--b--c};
			\coordinate (a) at (2,0); \coordinate (b) at (0.5,2.25); \coordinate (c) at (3.15,2.76); \draw pic[draw,angle radius=0.3cm,"$\theta$",right] {angle=a--b--c};
			\coordinate (a) at (3,0); \coordinate (b) at (2,0); \coordinate (c) at (3.15,2.76); \draw pic[draw,angle radius=0.3cm,"$\theta$",right] {angle=a--b--c};
			\node[black, below left] at (-0.5,2.8) {$\mathcal{C}_k$};
			\node[black, below left] at (-2.2,2.8) {$\mathcal{C}_{k+1}$};
	\end{tikzpicture}}
	\subfloat[]{\begin{tikzpicture}[rotate=25, scale=1.2]
			\draw[thick] (1,-1.5)circle(1.80277563773);
			\draw[thick] (0.5,-2.25)circle(2.7041634566);
			\draw[thick] (0,0)--(2,0); \draw[thick](2,0)--(1,-1.5);
			\draw[thick](1,-1.5)--(0,0);
			\draw[thick](0.5,-2.25)--(3.15,-2.76);
			\draw[thick](2,0)--(3.15,-2.76);
			\node[below,blue] at (0,0) {$R_{k-1}$};
			\filldraw[blue] (0,0) circle (2.0pt);
			\node[above,right,magenta] at (1,-1.5) {$C_k$};
			\filldraw[magenta] (1,-1.5) circle (2.0pt);
			 \node[above,magenta] at (0.5,-2.25) {$C_{k+1}$};
			 \filldraw[magenta] (0.5,-2.25) circle (2.0pt); \node[right,below,blue] at (2,0) {$R_k$};
			\filldraw[blue] (2,0) circle (2.0pt);
			\node[right,below,blue] at (3.55,-2.76) {$R_{k+1}$};
			\filldraw[blue] (3.15,-2.76) circle (2.0pt);
			\node[right,cyan] at (36/13,-24/13) {$S_{k+1}$};
			\filldraw[yellow] (36/13,-24/13) circle (2.0pt);
			\draw[thick,dashed] (1,-1.5)--(1,0);
			\draw[thick,dashed] (1,-1.5)--(0.5,-2.25);
			\draw[thick,dashed] (2,0)--(3,0);
			\draw[thick,dashed] (3.15,-2.76)--(3.6,-3.84);
			\coordinate (a) at (0,0); \coordinate (b) at (1,-1.5); \coordinate (c) at (2,0); \draw pic[draw,angle radius=0.3cm,"$\pi+\theta$",below] {angle=a--b--c};
			\coordinate (a) at (2,0); \coordinate (b) at (0.5,-2.25); \coordinate (c) at (3.15,-2.76); \draw pic[draw,angle radius=0.3cm,"$\pi+\theta$",right] {angle=a--b--c};
			\coordinate (a) at (3,0); \coordinate (b) at (2,0); \coordinate (c) at (3.15,-2.76); \draw pic[draw,angle radius=0.3cm,"$\pi+\theta$",right] {angle=a--b--c};
			\node[black, below left] at (-0.5,-2.8) {$\mathcal{C}_k$};
			\node[black, below left] at (-2.2,-2.8) {$\mathcal{C}_{k+1}$};
	\end{tikzpicture}}
	
	\caption{The construction of $\mathcal{C}_k,\,\mathcal{C}_{k+1},\,\mathcal{C}_2$. (a) $C_{k+1}$ lies entirely outside $C_k$ and it only touches it at $R_k$. (b) $C_{k+1}$ lies entirely outside $C_k$ and it only touches it at $R_k$ when$\theta^*$ lies in the third or fourth quadrants.}
	\label{An2} 
\end{figure*}
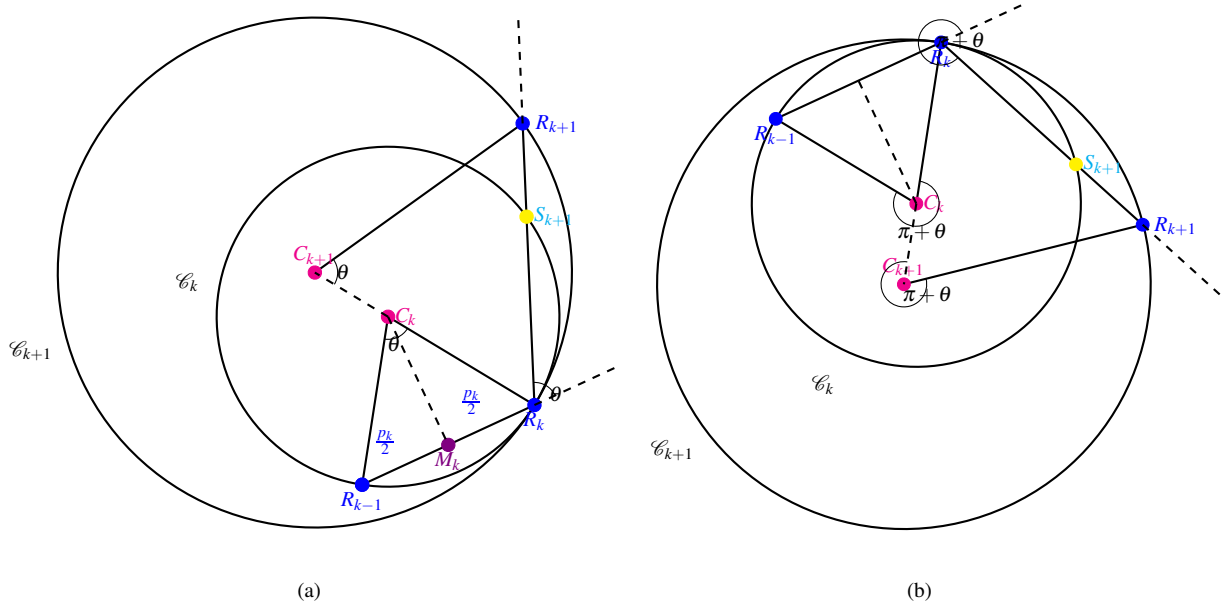

Lemma \ref{lem:1.1} remains true when $\theta=0$, because in this case 
\begin{equation}\label{eq:3.15}
	R_0=p_0>0, \ R_n=\sum\limits_{j=0}^{n}p_j>0,
\end{equation}
i.e. all points lie on the positive part of the $X$-axis. The circles  $\mathcal{C}_0,\;\mathcal{C}_1,\cdots, \mathcal{C}_n$ are reduced to line segments. Nevertheless, $R_n\ne0$. The lemma is also true if $\theta=\pi$, because
\begin{equation}\label{eq:3.16}
	R_0=p_0>0, \ R_n=p_0-p_1+p_2-\cdots+(-1)^np_n.
\end{equation}
When $n=2m+1$ is odd, $m \geq 0$, then 
\begin{equation}\label{eq:3.17}
	R_n=(p_0-p_1)+(p_2-p_3)+\cdots+(p_{2m}-p_{2m+1})<0,
\end{equation}
i.e. $R_n\ne0$. Here the points $R_j$ lie alternately on the positive and negative parts of the $X$-axis. If $n=2m$ is even, $m \geq 1$, then 
\begin{equation}\label{eq:3.18}
	R_n=p_0+(p_2-p_1)+\cdots+(p_{2m}-p_{2m-1})>0.
\end{equation}
i.e. $R_n\ne0$. Also here the points $R_j$ lie alternately on the positive and negative parts of the $X$-axis. Thus the lemma is true if $\theta=\pi$. The proofs when $\theta$ lies in the third or fourth quadrants results from the previous ones. According to Figure \ref{An2} (b), we can demonstrate a construction of $R_n$ such that Lemma \ref{lem:1.1} is true for $\theta^*=\pi +\theta$, $0<\theta<\pi$ and deduce geometrically that
\begin{equation}
p_0+p_1e^{i\theta^*}+p_2e^{2i\theta^*}+\cdots+p_ne^{in\theta^*}\ne 0.
\end{equation}

\begin{rem}
The points $R_{k-1},\,C_{k-1},\,C_k$ are collinear. This is shown for $\mathcal{C}_0,\,\mathcal{C}_1,\,\mathcal{C}_2$ in the previous section by construction and by \textup{(\ref{eq:2.4})}. By construction, this is also true for  $R_{k-1},\,C_{k-1},\,C_k,\; k\ge2$. Moreover, let

\begin{equation}\label{eq:2.8A}
\Delta_k=\begin{vmatrix}
		x_{k-1} & y_{k-1}  & 1 \\
		h_{k-1} & l_{k-1} & 1 \\
		h_{k} & l_{k} & 1 
	\end{vmatrix},\quad k\ge2.
\end{equation}
From \textup{(\ref{A1})-(\ref{A2})} and \textup{(\ref{A11})-(\ref{A12})}, we get
\begin{eqnarray}\label{A111}
h_k&=&
x_{k-1}+\frac{p_k}{2}\cos k\theta-\frac{p_k}{2}\sin k\theta \cot \frac{\theta}{2},\\
\label{A112}
l_k&=& y_{k-1}+\frac{p_k}{2}\sin k\theta+\frac{p_k}{2}\cos k\theta \cot \frac{\theta}{2}.
\end{eqnarray}
Consequently, using the properties of determinants, we obtain 
\begin{equation}\label{eq:2.8A1}
\Delta_k=\begin{vmatrix}
x_{k-1} &\ y_{k-1}  & 1 \\
-\frac{p_{k-1}}{2}\cos(k-1)\theta-\frac{p_{k-1}}{2}\sin(k-1)\theta \cot \frac{\theta}{2} & \ -\frac{p_{k-1}}{2}\sin(k-1)\theta+\frac{p_{k-1}}{2}\cos(k-1)\theta \cot \frac{\theta}{2} & 0 \\
\frac{p_k}{2}\cos k\theta-\frac{p_k}{2}\sin k\theta \cot \frac{\theta}{2} &\ \frac{p_k}{2}\sin k\theta+\frac{p_k}{2}\cos k\theta \cot \frac{\theta}{2} & 0
\end{vmatrix}=0.
\end{equation}
\end{rem}

In the following section, we introduce another geometrically-based proof of Enestr\"om-Kakeya's theorem. This approach is based on the following lemme.
\begin{lem}\label{lem:3.1}
If $0<q_n\leq q_{n-1}\leq \cdots \leq q_1\leq q_0 $ are positive numbers, $\theta \in \Bbb R$, $n \in \Bbb N_0$, then
\begin{equation}\label{A14}
Q_n=q_0+q_1e^{i\theta}+q_2e^{2i\theta}+\cdots+q_ne^{in\theta}\ne 0.
\end{equation}
\end{lem}
\begin{proof} Since
\begin{equation}\label{A15}
e^{-in\theta}Q_n=q_0e^{-in\theta}+q_1e^{-i(n-1)\theta}+q_2e^{-i(n-2)\theta}+\cdots+q_{n-1}e^{-i\theta}+q_n,
\end{equation}
then, Lemma \ref{lem:1.1} implies that $e^{-in\theta}Q_n \ne 0,$ proving Lemma \ref{lem:3.1}.\qed
\end{proof}
Using a similar technique of Kakeya, we can prove Enestr\"om-Kakeya's theorem using Lemma \ref{lem:3.1}. However, the proof of Lemma \ref{lem:3.1} depends entirely on Lemma \ref{lem:1.1}, and therefore adds no new. To have a completely independent proof, we can use a technique similar to the geometric structure  (\ref{eq:1.5}), which is done in the next section.

\section{Internally interlacing construction}

 Keeping previous notations, we prove that the points $Q_0=q_0$, $Q_k=\sum_{j=0}^k q_je^{ij\theta}$,$k=1,2,\cdots,n$ lie on the internally interlacing circles $\mathfrak{C}_0,\;\mathfrak{C}_1,\cdots, \mathfrak{C}_n$ with centers and radii $O_0, O_1,\cdots, O_n$ and $\rho_0, \rho_1,\cdots, \rho_n$ respectively, for which 
\begin{equation}\label{A16}
\mathfrak{C}_{k-1}\cap\mathfrak{C}_k=\{Q_{k-1}\}, \ \rho_{k-1}-\rho_k=\|O_{k-1}-O_k\|,
\end{equation}
provided that the $q_j$'s are distinct and $\theta\ne0,\pi.$

\begin{figure*}[!ht]
	\centering
	\subfloat[]{\begin{tikzpicture}[scale=1.2]
			\draw[thick] (1.5,1.95483805926)circle(2.46401944756);
			\draw[thick] (1.75, 1.62903171605)circle(2.05334953963);
			\draw[thick] (2.5088190451, 1.94334485592)circle(1.23200972378);
			\draw[thick] (0,0)--(3,0) (0,0)--(1.5,1.95483805926) (1.5,1.95483805926)--(3,0) (3,0)--(3.64704761276,2.41481456572) (1.75, 1.62903171605)--(3.64704761276,2.41481456572) (2.5088190451, 1.94334485592)--(2.34800950708,3.16481456572) (3.64704761276,2.41481456572)--(2.34800950708,3.16481456572);
			\draw[thick, blue] (1.5,1.95483805926)--(3,0);
			\draw[thick, magenta] (1.75, 1.62903171605)--(3.64704761276,2.41481456572);
			\node[below] at (0,0) {$O$};
			\filldraw[black] (0,0) circle (2.0pt);
			\node[below,blue] at (3,0) {$Q_0$};
			\filldraw[blue] (3,0) circle (2.0pt);
			 \node[right,above,magenta] at (1.5,1.95483805926) {$O_0$};
			 \filldraw[magenta] (1.5,1.95483805926) circle (2.0pt);
			  \node[right,blue] at (3.64704761276,2.41481456572) {$Q_1$};
			  \filldraw[blue] (3.64704761276,2.41481456572) circle (2.0pt);
			\node[above,blue] at ( 2.34800950708,3.16481456572) {$Q_2$};
			\filldraw[blue] ( 2.34800950708,3.16481456572) circle (2.0pt);
			\node[below,magenta] at (1.75, 1.62903171605) {$O_1$};
			\filldraw[magenta](1.75, 1.62903171605) circle (2.0pt);
			\node[below,magenta] at (2.5088190451, 1.94334485592) {$O_2$};
			\filldraw[magenta](2.5088190451, 1.94334485592) circle (2.0pt);
			\node[right,cyan] at (3.77645713531, 2.89777747887) {$S_1$};
			\filldraw[yellow](3.77645713531, 2.89777747887) circle (2.0pt);
			\node[above,cyan] at (1.4819841033,3.66481456572) {$S_2$};
				\filldraw[yellow](1.4819841033,3.66481456572) circle (2.0pt);
			\draw[thick, dashed] (3,0)--(5,0) (3.64704761276,2.41481456572)--(3.77645713531, 2.89777747887) (2.34800950708,3.16481456572)--(1.4819841033,3.66481456572);
			\coordinate (a) at (0,0); \coordinate (b) at (1.5,1.95483805926); 
			\coordinate (c) at (3,0); \draw pic[draw,angle radius=0.3cm,"$\theta$",below] {angle=a--b--c};
			\coordinate (a) at (3,0); \coordinate (b) at (1.75, 1.62903171605); 
			\coordinate (c) at (3.64704761276,2.41481456572); \draw pic[draw,angle radius=0.3cm,"$\theta$",right] {angle=a--b--c};
			\coordinate (a) at (4,0); \coordinate (b) at (3,0); 
			\coordinate (c) at (3.64704761276,2.41481456572); \draw pic[draw,angle radius=0.3cm,"$\theta$",right] {angle=a--b--c};
			\coordinate (a) at (3.77645713531, 2.89777747887); \coordinate (b) at (3.64704761276,2.41481456572); 
			\coordinate (c) at (2.34800950708,3.16481456572); \draw pic[draw,angle radius=0.3cm,"$\theta$",above] {angle=a--b--c};
			\coordinate (a) at (3.64704761276,2.41481456572); \coordinate (b) at (2.5088190451, 1.94334485592); 
			\coordinate (c) at (2.34800950708,3.16481456572); \draw pic[draw,angle radius=0.3cm,"$\theta$",above] {angle=a--b--c};
			\node[black,below left] at (1,2.5) {$\mathfrak{C}_{2}$};
			\node[black,below left] at (0,3.2) {$\mathfrak{C}_{1}$};
			\node[black,below left] at (-0.5,3.8) {$\mathfrak{C}_{0}$};
	\end{tikzpicture}}
	\subfloat[]{\begin{tikzpicture}[scale=1.2]
			\draw[thick] (1.5,1.95483805926)circle(2.46401944756);
			\draw[thick] (1.75, 1.62903171605)circle(2.05334953963);
			\draw[thick] (0,0)--(3,0) (0,0)--(1.5,1.95483805926) (1.5,1.95483805926)--(3,0) (3,0)--(3.64704761276,2.41481456572) (1.75, 1.62903171605)--(3.64704761276,2.41481456572) (1.75, 1.62903171605)--(1.4819841033,3.66481456572) (3.64704761276,2.41481456572)--(1.4819841033,3.66481456572 );
			\draw[thick, blue] (1.5,1.95483805926)--(3,0);
			\node[below] at (0,0) {$O$};
			\filldraw[black](0,0) circle (2.0pt);
			\node[below,blue] at (3,0) {$Q_0$};
			\filldraw[blue](3,0) circle (2.0pt);
			 \node[right,above,magenta] at (1.5,1.95483805926) {$O_0$}; 
			 \filldraw[magenta](1.5,1.95483805926) circle (2.0pt);
			 \node[right,blue] at (3.64704761276,2.41481456572) {$Q_1$};
			 \filldraw[blue](3.64704761276,2.41481456572) circle (2.0pt);
			\node[above,blue] at (1.4819841033,3.66481456572 ) {$Q_2=S_2$};
			 \filldraw[blue](1.4819841033,3.66481456572) circle (2.0pt);
			\node[below,magenta] at (1.75, 1.62903171605) {$O_1$};
			\filldraw[magenta](1.75, 1.62903171605) circle (2.0pt);
			\node[right,cyan] at (3.77645713531, 2.89777747887) {$S_1$};
			\filldraw[yellow](3.77645713531, 2.89777747887) circle (2.0pt);
			\draw[thick, dashed] (3,0)--(5,0) (3.64704761276,2.41481456572)--(3.77645713531, 2.89777747887);
			\coordinate (a) at (0,0); \coordinate (b) at (1.5,1.95483805926); 
			\coordinate (c) at (3,0); \draw pic[draw,angle radius=0.3cm,"$\theta$",below] {angle=a--b--c};
			\coordinate (a) at (3,0); \coordinate (b) at (1.75, 1.62903171605); 
			\coordinate (c) at (3.64704761276,2.41481456572); \draw pic[draw,angle radius=0.3cm,"$\theta$",right] {angle=a--b--c};
			\coordinate (a) at (4,0); \coordinate (b) at (3,0); 
			\coordinate (c) at (3.64704761276,2.41481456572); \draw pic[draw,angle radius=0.3cm,"$\theta$",right] {angle=a--b--c};
			\coordinate (a) at (3.64704761276,2.41481456572); \coordinate (b) at (1.75, 1.62903171605); 
			\coordinate (c) at (1.4819841033,3.66481456572); \draw pic[draw,angle radius=0.3cm,"$\theta$",above] {angle=a--b--c};
			\node[black,below left] at (0.2,3.2) {$\mathfrak{C}_{1}=\mathfrak{C}_{2}$};
			\node[black,below left] at (-0.5,3.8) {$\mathfrak{C}_{0}$};
	\end{tikzpicture}}
	
	\caption{The construction of $\mathfrak{C}_0,\,\mathfrak{C}_1,\,\mathfrak{C}_2$ when $\theta=\frac{5\pi}{12}$. (a) $q_0=3, q_1=2.5, q_2=1.5$ are different. (b) $q_0=3, q_1=q_2=2.5$ and $\mathfrak{C}_1,\,\mathfrak{C}_2$ are identical.}
	\label{An3} 
\end{figure*}
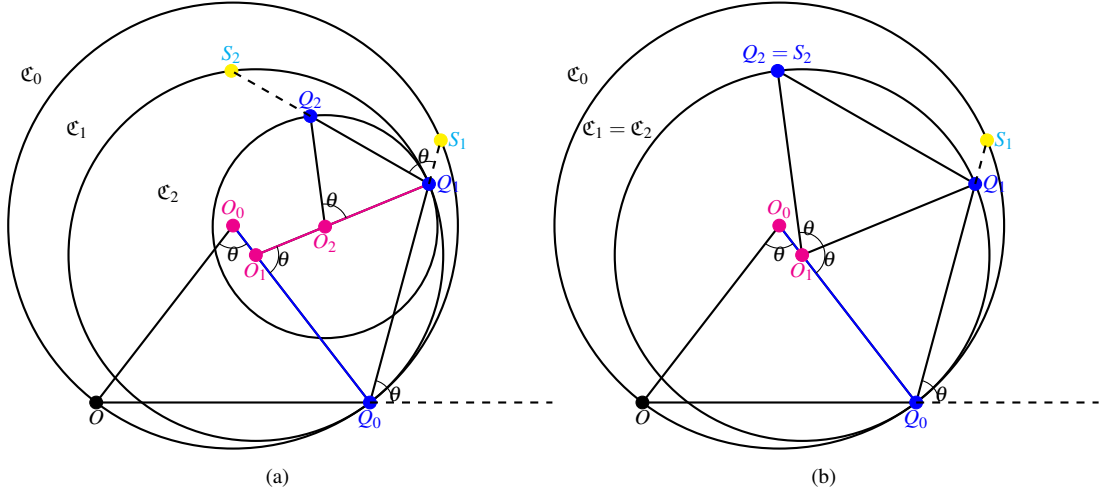

\begin{figure*}[!ht]
	\centering
	\subfloat[]{\begin{tikzpicture}[scale=1.4]
			\draw[thick] (1.5,1.5)circle(2.12132034356);
			\draw[thick] (1.75,1.25)circle(1.76776695297);
			\draw[thick] (2.25,1.75)circle(1.06066017178);
			\draw[thick] (0,0)--(3,0) (0,0)--(1.5,1.5) (1.5,1.5)--(3,0) (3,0)--(3,2.5) (3,2.5)--(1.5,2.5) (2.25,1.75)--(1.5, 2.5);
			\draw[thick, blue] (1.5,1.5)--(3,0);
			\draw[thick, magenta] (1.75,1.25)--(3, 2.5);
			\node[below] at (0,0) {$O$};
			\filldraw[black](0, 0) circle (2.0pt);
			\node[below,blue] at (3,0) {$Q_0$};
			\filldraw[blue](3, 0) circle (2.0pt);
			 \node[right,above,magenta] at (1.5,1.5) {$O_0$};
			 \filldraw[magenta](1.5, 1.5) circle (2.0pt);
			  \node[right,blue] at (3,2.5) {$Q_1$};
			  \filldraw[blue](3,2.5) circle (2.0pt);
			\node[above,blue] at (1.5,2.5) {$Q_2$};
			\filldraw[blue](1.5,2.5) circle (2.0pt);
			\node[below,magenta] at (1.75, 1.25) {$O_1$};
			\filldraw[magenta](1.75, 1.25) circle (2.0pt);
			\node[right,magenta] at (2.25, 1.75) {$O_2$};
			\filldraw[magenta](2.25, 1.75) circle (2.0pt);
			\node[right,cyan] at (3,3) {$S_1$};
			\filldraw[yellow](3, 3) circle (2.0pt);
			\node[above,cyan] at (0.5, 2.5) {$S_2$};
			\filldraw[yellow](0.5,2.5) circle (2.0pt);
			\draw[thick, dashed] (3,0)--(4.5,0) (3,2.5)--(3,3) (0.5,2.5)--(1.5, 2.5);
			\coordinate (a) at (0,0); \coordinate (b) at (1.5,1.5); 
			\coordinate (c) at (3,0); \draw pic[draw,angle radius=0.3cm,"$\theta$",below] {angle=a--b--c};
			\coordinate (a) at (3,0); \coordinate (b) at (1.75, 1.25); 
			\coordinate (c) at (3, 2.5); \draw pic[draw,angle radius=0.3cm,"$\theta$",right] {angle=a--b--c};
			\coordinate (a) at (4,0); \coordinate (b) at (3,0); 
			\coordinate (c) at (3,2.5); \draw pic[draw,angle radius=0.3cm,"$\theta$",right] {angle=a--b--c};
			\coordinate (a) at (3,3); \coordinate (b) at (3, 2.5); 
			\coordinate (c) at (0.5, 2.5); \draw pic[draw,angle radius=0.3cm,"$\theta$",above] {angle=a--b--c};
			\coordinate (a) at (3,2.5); \coordinate (b) at (2.25, 1.75); 
			\coordinate (c) at (1.5, 2.5); \draw pic[draw,angle radius=0.3cm,"$\theta$",above] {angle=a--b--c};
			\node[black,below left] at (1.2,2) {$\mathfrak{C}_{2}$};
			\node[black,below left] at (0.5,2.7) {$\mathfrak{C}_{1}$};
			\node[black,below left] at (-0.3,3.3) {$\mathfrak{C}_{0}$};
	\end{tikzpicture}}
	\subfloat[]{\begin{tikzpicture}[scale=1.4]
			\draw[thick] (1.5,1.5)circle(2.12132034356);
			\draw[thick] (1.75,1.25)circle(1.76776695297);
			\draw[thick] (0,0)--(3,0) (0,0)--(1.5,1.5) (1.5,1.5)--(3,0) (3,0)--(3,2.5) (3, 2.5)--(0.5, 2.5);
			\draw[thick, blue] (1.5,1.5)--(3,0) (0.5,2.5)--(1.5,1.5);
			\draw[thick, magenta] (1.75,1.25)--(3, 2.5);
			\node[below] at (0,0) {$O$};
			\filldraw[black](0, 0) circle (2.0pt);
			\node[below,blue] at (3,0) {$Q_0$};
			\filldraw[blue](3, 0) circle (2.0pt); \node[right,above,magenta] at (1.5,1.5) {$O_0$};
			\filldraw[magenta](1.5, 1.5) circle (2.0pt); \node[right,blue] at (3,2.5) {$Q_1$};
			\filldraw[blue](3, 2.50) circle (2.0pt);
			\node[below,magenta] at (1.75, 1.25) {$O_1$};
			\filldraw[magenta](1.75, 1.25) circle (2.0pt);
			\node[right,cyan] at (3,3) {$S_1$};
			\filldraw[yellow](3, 3) circle (2.0pt);
			\node[above,blue] at (0.5, 2.5) {$Q_2=S_2$};
			\filldraw[blue](0.5,2.5) circle (2.0pt);
			\draw[thick, dashed] (3,0)--(4.5,0) (3,2.5)--(3,3);
			\coordinate (a) at (0,0); \coordinate (b) at (1.5,1.5); 
			\coordinate (c) at (3,0); \draw pic[draw,angle radius=0.3cm,"$\theta$",below] {angle=a--b--c};
			\coordinate (a) at (3,0); \coordinate (b) at (1.75, 1.25); 
			\coordinate (c) at (3, 2.5); \draw pic[draw,angle radius=0.3cm,"$\theta$",right] {angle=a--b--c};
			\coordinate (a) at (4,0); \coordinate (b) at (3,0); 
			\coordinate (c) at (3,2.5); \draw pic[draw,angle radius=0.3cm,"$\theta$",right] {angle=a--b--c};
			\coordinate (a) at (3, 2.5); \coordinate (b) at (1.75, 1.25); 
			\coordinate (c) at (0.5,2.5); \draw pic[draw,angle radius=0.2cm,"$\theta$",above] {angle=a--b--c};
			\node[black,below left] at (0.5,2.4) {$\mathfrak{C}_{1}=\mathfrak{C}_{2}$};
			\node[black,below left] at (-0.3,3.3) {$\mathfrak{C}_{0}$};
	\end{tikzpicture}}
	
	\caption{The construction of $\mathfrak{C}_0,\,\mathfrak{C}_1,\,\mathfrak{C}_2$ when $\theta=\frac{\pi}{2}$. (a) $q_0=3, q_1=2.5, q_2=1.5$ are different. (b) $q_0=3, q_1=q_2=2.5$ and $\mathfrak{C}_1,\,\mathfrak{C}_2$ are identical.}
	\label{An4} 
\end{figure*}
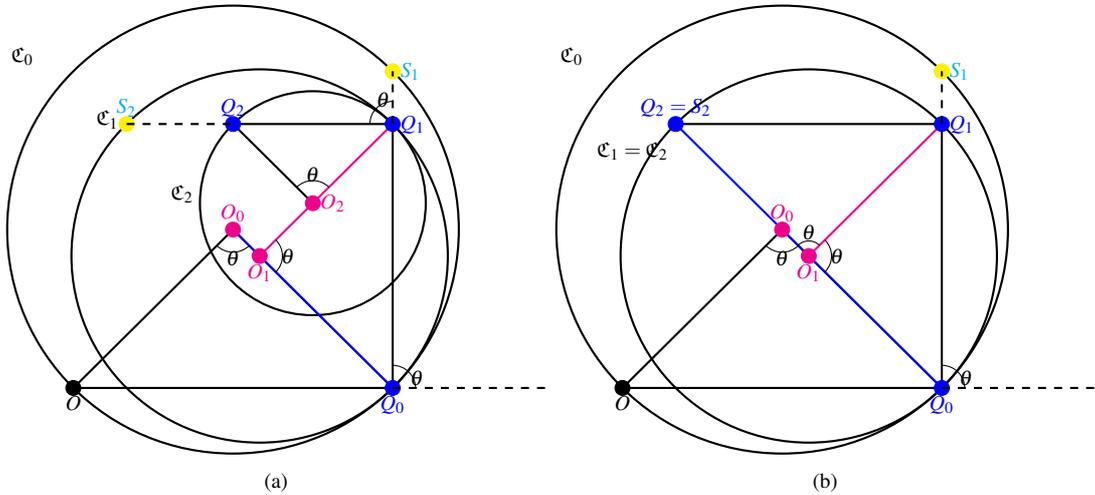


As is done before, we discuss the construction for $\mathfrak{C}_0,\,\mathfrak{C}_1,\,\mathfrak{C}_2$. We now assume without any loss of generality that $0<\theta<\pi$.  The circle $\mathfrak{C}_0$ is the one that passes through the origin $O$, the point $Q_0=(q_0,0)$ and $S_1=\overrightarrow{OQ_0}+\overrightarrow{Q_0S_1}$, $\overrightarrow{Q_0S_1}=q_0e^{i\theta}$. Thus $S_1=(q_0+q_0\cos\theta, q_0\sin \theta)$.  Therefore, the center and the radius of $\mathfrak{C}_0$ are 
\begin{equation}\label{eq:4.2}
O_0\left(\frac{q_0}{2}, \frac{q_0}{2}\cot \frac{\theta}{2}\right), \ \rho_0= \frac{q_0}{2}\csc\frac{\theta}{2}.
\end{equation}
If $q_0=q_1$, then $Q_1$ is merely $S_1$, and $\mathfrak{C}_1$ coincides with $\mathfrak{C}_0$. Otherwise, $Q_0S_1$ is reduced   to $Q_0Q_1$, where $\overrightarrow{Q_0S_1}\parallel \overrightarrow{Q_0Q_1}$ and $\|Q_0Q_1\|=q_1$. Thus $Q_1=(q_0+q_1\cos \theta, q_1\sin \theta)=\overrightarrow{OQ_1}$. The circle $\mathfrak{C}_1$ is determined by $Q_0, Q_1$ and the central angle $\measuredangle Q_0O_1Q_1$, which is taken to be of measure $\theta$. Thus
\begin{equation}\label{eq:4.3}
O_1\left(q_0-\frac{q_1}{2}, \frac{q_1}{2}\cot \frac{\theta}{2}\right), \ \rho_1= \frac{q_1}{2}\csc\frac{\theta}{2}.
\end{equation}
The circle $\mathfrak{C}_1$ lies entirely inside $\mathfrak{C}_0$ and it only touches it at $Q_0$ because 
\begin{equation}\label{eq:4.4}
\rho_0-\rho_1=\|O_0-O_1\|=\frac{q_0-q_1}{2}\csc\frac{\theta}{2}.
\end{equation}
Also $Q_0,\,O_0,\,O_1$ are collinear. The construction of $\mathfrak{C}_2$ is similar to that of $\mathfrak{C}_1$. We first locate $S_2$ whose position vector is 
\begin{equation}\label{eq:4.5}
\overrightarrow{OS_2}=\overrightarrow{OQ_0}+\overrightarrow{Q_0Q_1} +\overrightarrow{Q_1S_2},\quad  \overrightarrow{Q_1S_2}= q_1e^{2i\theta}.
\end{equation}
Thus the coordinates of $S_2$ are 
\begin{equation}\label{eq:4.6}
S_2\left(q_0+q_1\cos \theta+q_1\cos 2\theta, q_1\sin \theta+q_1\sin 2\theta \right).
\end{equation}
Again if $q_1=q_2$, then $S_2=Q_2$, and $\mathfrak{C}_2$ coincides with $\mathfrak{C}_1$. Otherwise, we restrict $\overrightarrow{Q_1S_2}$ to $\overrightarrow{Q_1Q_2}$,
\begin{equation}\label{eq:4.7}
\overrightarrow{Q_1Q_2}\parallel \overrightarrow{Q_1S_2}, \quad \|\overrightarrow{Q_1Q_2}\|=q_2.
\end{equation}
The circle $\mathfrak{C}_2$ is determined by $Q_1$ and $Q_2$ and the central angle $\measuredangle Q_1O_2Q_2$, which is taken to satisfy $\measuredangle {Q_1O_1Q_2}=\theta$. 
Hence
\begin{equation}\label{eq:4.8}
O_2\left(q_0+q_1\cos \theta+\frac{q_2}{2}\cos 2\theta- \frac{q_2}{2}\sin 2\theta \cot \frac{\theta}{2}, q_1\sin \theta+\frac{q_2}{2}\sin 2\theta+ \frac{q_2}{2}\cos 2\theta \cot \frac{\theta}{2}\right), \ \rho_2= \frac{q_2}{2}\csc\frac{\theta}{2},
\end{equation}
and $Q_1, O_1, O_2$ are collinear. Moreover, 
\begin{equation}\label{eq:4.9}
\rho_1-\rho_2=\|O_1-O_2\|=\frac{q_1-q_2}{2}\csc\frac{\theta}{2},
\end{equation}
i.e. $\mathfrak{C}_1$ lies entirely outside $\mathfrak{C}_2$ and it touches it only at $Q_2$.

Figures \ref{An3}-\ref{An5} exhibit the constructions of  $\mathfrak{C}_0,\,\mathfrak{C}_1,\,\mathfrak{C}_2$ for different values of $\theta$. Figure \ref{An3} illustrates the construction of $\mathfrak{C}_0,\,\mathfrak{C}_1,\,\mathfrak{C}_2$ when $\theta=\frac{5\pi}{12}$ is acute. We considered two cases when $q_0,\,q_1,\,q_2$ are all different and when two values coincide to each others. The cases when $\theta$ is a right or an obtuse angle are depicted in Figure \ref{An4} and Figure \ref{An5} respectively. The construction coincide with our calculations.


\begin{figure*}[!ht]
	\centering
	\subfloat[]{\begin{tikzpicture}[scale=1.7]
			\draw[thick] (1.5,0.86602540378)circle(1.73205080757);
			\draw[thick] (1.75,0.72168783648)circle(1.44337567297);
			\draw[thick] (1.75,1.29903810568)circle(0.86602540378);
			\draw[thick] (0,0)--(3,0) (0,0)--(1.5,0.86602540378) (1.5,0.86602540378)--(3,0) (3,0)--(1.75,2.16506350946) (1.75,2.16506350946)--(1.75, 0.72168783648) (1,0.86602540378)--(1.75,2.16506350946) (1.75, 1.29903810568)--(1,0.86602540378);
			\draw[thick, blue] (1.5,0.86602540378)--(3,0);
			\draw[dashed] (0.5,0)--(1,0.86602540378);
			\draw[thick, magenta] (1.75, 0.72168783648)--(1.75,2.16506350946);
			\node[below] at (0,0) {$O$};
			\filldraw[black](0,0) circle (1.5pt);
			\node[below,blue] at (3,0) {$Q_0$};
			\filldraw[blue](3, 0) circle (2.0pt);
			 \node[left,magenta] at (1.5,0.86602540378) {$O_0$};
			 \filldraw[magenta](1.5,0.86602540378) circle (1.5pt); \node[above,right,blue] at (1.75,2.16506350946) {$Q_1$};
			 \filldraw[blue](1.75,2.16506350946) circle (1.5pt);
			\node[left,blue] at (1,0.86602540378) {$Q_2$};
			\filldraw[blue](1,0.86602540378) circle (1.5pt);
			\node[below,magenta] at (1.75, 0.72168783648) {$O_1$};
			\filldraw[magenta](1.75, 0.72168783648) circle (1.5pt);
			\node[right,magenta] at (1.75, 1.29903810568) {$O_2$};
			\filldraw[magenta](1.75, 1.29903810568) circle (1.5pt);
			\node[above,cyan] at (1.5, 2.59807621135) {$S_1$};
			\filldraw[yellow](1.5, 2.59807621135) circle (1.5pt);
			\node[below,cyan] at (0.5, 0) {$S_2$};
			\filldraw[yellow](0.5, 0) circle (1.5pt);
			\draw[thick, dashed] (3,0)--(4.2,0) (1.5, 2.59807621135)--(1.75,2.16506350946);
			\coordinate (a) at (0,0); \coordinate (b) at (1.5,0.86602540378); 
			\coordinate (c) at (3,0); \draw pic[draw,angle radius=0.2cm,"$\theta$",below] {angle=a--b--c};
			\coordinate (a) at (3,0); \coordinate (b) at (1.75, 0.72168783648); 
			\coordinate (c) at (1.75,2.16506350946); \draw pic[draw,angle radius=0.2cm,"$\theta$",right] {angle=a--b--c};
			\coordinate (a) at (4,0); \coordinate (b) at (3,0); 
			\coordinate (c) at (1.75,2.16506350946); \draw pic[draw,angle radius=0.2cm,"$\theta$",right] {angle=a--b--c};
			\coordinate (a) at (1.5, 2.59807621135); \coordinate (b) at (1.75,2.16506350946); 
			\coordinate (c) at (1,0.86602540378); \draw pic[draw,angle radius=0.2cm,"$\theta$",left] {angle=a--b--c};
			\coordinate (a) at (1.75,2.16506350946); \coordinate (b) at (1.75, 1.29903810568); 
			\coordinate (c) at (1,0.86602540378); \draw pic[draw,angle radius=0.2cm,"$\theta$",left] {angle=a--b--c};
			\node[black,below left] at (0.9,1.5) {$\mathfrak{C}_{2}$};
			\node[black,below left] at (0.6,2) {$\mathfrak{C}_{1}$};
			\node[black,below left] at (0,2.4) {$\mathfrak{C}_{0}$};
	\end{tikzpicture}}
	\subfloat[]{\begin{tikzpicture}[scale=1.7]
			\draw[thick] (1.5,0.86602540378)circle(1.73205080757);
			\draw[thick] (1.75, 0.72168783648)circle(1.44337567297);
			\draw[thick] (0,0)--(3,0) (0,0)--(1.5,0.86602540378) (1.5,0.86602540378)--(3,0) (3,0)--(1.75,2.16506350946) (1.75,2.16506350946)--(1.75, 0.72168783648);
			\draw[thick, blue] (1.5,0.86602540378)--(3,0);
			\draw[thick] (0.5, 0)--(1.75,2.16506350946);
			\draw[thick, magenta] (1.75, 0.72168783648)--(1.75,2.16506350946);
			\node[below] at (0,0) {$O$};
			\filldraw[black](0, 0) circle (1.5pt);
			\node[below,blue] at (3,0) {$Q_0$};
			\filldraw[blue](3, 0) circle (1.5pt);
			 \node[left,magenta] at (1.5,0.86602540378) {$O_0$};
			 \filldraw[magenta](1.5,0.86602540378) circle (1.5pt);
			  \node[above,right,blue] at (1.75,2.16506350946) {$Q_1$};
			 \filldraw[blue](1.75,2.16506350946) circle (1.5pt);
			\node[below,magenta] at (1.75, 0.72168783648) {$O_1=O_2$};
			\filldraw[magenta](1.75, 0.72168783648) circle (1.5pt);
			\node[above,cyan] at (1.5, 2.59807621135) {$S_1$};
			\filldraw[yellow](1.5, 2.59807621135) circle (1.5pt);
			\node[below,blue] at (0.5, 0) {$S_2=Q_2$};
			\filldraw[blue](0.5, 0) circle (1.5pt);
			\draw[thick, dashed] (3,0)--(4.2,0) (1.5, 2.59807621135)--(1.75,2.16506350946);
			\coordinate (a) at (0,0); \coordinate (b) at (1.5,0.86602540378); 
			\coordinate (c) at (3,0); \draw pic[draw,angle radius=0.2cm,"$\theta$",below] {angle=a--b--c};
			\coordinate (a) at (3,0); \coordinate (b) at (1.75, 0.72168783648); 
			\coordinate (c) at (1.75,2.16506350946); \draw pic[draw,angle radius=0.2cm,"$\theta$",right] {angle=a--b--c};
			\coordinate (a) at (4,0); \coordinate (b) at (3,0); 
			\coordinate (c) at (1.75,2.16506350946); \draw pic[draw,angle radius=0.2cm,"$\theta$",right] {angle=a--b--c};
			\node[black,below left] at (0.8,2) {$\mathfrak{C}_{1}=\mathfrak{C}_{2}$};
			\node[black,below left] at (0,2.4) {$\mathfrak{C}_{0}$};
	\end{tikzpicture}}
	
	\caption{The construction of $\mathfrak{C}_0,\,\mathfrak{C}_1,\,\mathfrak{C}_2$ when $\theta=\frac{2\pi}{3}$. (a) $q_0=3, q_1=2.5, q_2=1.5$ are different. (b) $q_0=3, q_1=q_2=2.5$ and $\mathfrak{C}_1,\,\mathfrak{C}_2$ are identical.}
	\label{An5} 
\end{figure*}
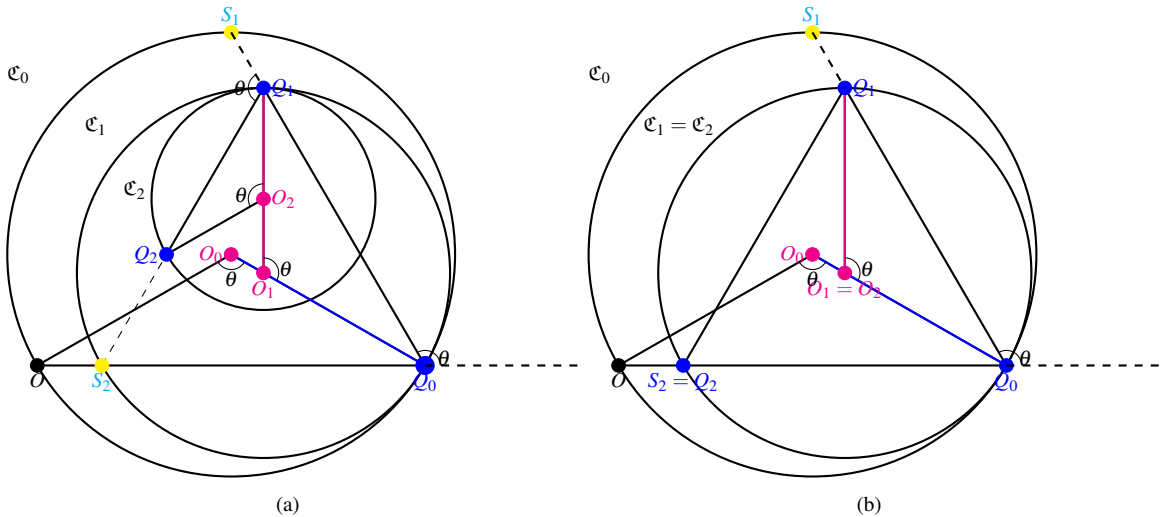


We proved (\ref{A16}) for $\mathfrak{C}_0,\,\mathfrak{C}_1,\,\mathfrak{C}_2$. Now we accomplish the prove for all $k.$ Indeed, let $k\ge2$. Similar to Kakeya structure, the circle $\mathfrak{C}_k$ is constructed to pass through $Q_{k-1},\,Q_k$ and $\measuredangle {Q_{k-1}O_kQ_k
}=\theta$. Thus $\mathfrak{C}_k$ is completely determined. From triangle $\triangle Q_{k-1}O_kQ_k$, cf. Figure \ref{fig11}. The radius of $\mathfrak{C}_k$ is
\begin{equation}\label{eq:4.10}
\rho_k=\frac{q_k}{2}\csc \frac{\theta}{2},
\end{equation}
and the center $O_k(f_k,g_k),\,k\ge2$ is determined via
\begin{eqnarray}\label{eq:4.11}
f_k&=&
\frac{u_{k-1}+u_k}{2}-\frac{v_k-v_{k-1}}2\,\cot \frac{\theta}{2}\\
\label{eq:4.12}
g_k&=& \frac{v_{k-1}+v_k}{2}+\frac{u_k-u_{k-1}}2\,\cot \frac{\theta}{2},
\end{eqnarray}
where $(u_{k}, v_{k})$ are the coordinates of $Q_k$,
\begin{eqnarray}\label{eq:4.13}
u_k&=&\left\{\begin{array}{ll}
q_0+q_1\cos\theta+\cdots+q_{k}\cos k\theta,& k\ge1,\\
{}&{}\\
q_0,&k=0,
\end{array}
\right.\\
\label{eq:4.14}
v_k&=&\left\{\begin{array}{ll}
q_1\sin\theta+\cdots+q_{k}\sin k\theta,& k\ge1,\\
		{}&{}\\
		0,&k=0.
	\end{array}
	\right.
\end{eqnarray}
Using (\ref{eq:4.13})-(\ref{eq:4.14}) , we obtain
\begin{eqnarray}\label{eq:4.15}
	f_k&=&
	q_0+q_1\cos\theta+\cdots+q_{k-1}\cos(k-1)\theta+\frac{q_k}{2}\cos k\theta-\frac{q_k}{2}\sin k\theta \cot \frac{\theta}{2},\\
	\label{eq:4.16}
	g_k&=& q_1\sin\theta+\cdots+q_{k-1}\sin(k-1)\theta+\frac{q_k}{2}\sin k\theta+\frac{q_k}{2}\cos k\theta \cot \frac{\theta}{2}.
\end{eqnarray}
Hence, the circles $\mathfrak{C}_0,\,\mathfrak{C}_1,\cdots,\mathfrak{C}_n$ are all determined. Using the techniques demonstrated above, the construction (\ref{A16}) can be analogously proved, and  the proof of Lemma \ref{lem:3.1} is complete.

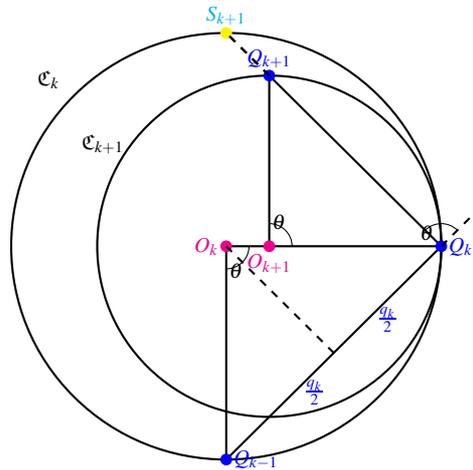
\begin{figure}
\centering\begin{tikzpicture}[scale=0.8, rotate=-45]
\draw[thick] (3.5,2.5)circle(3.53553390593);
\draw[thick] (4,3)circle(2.82842712475);
\draw[thick] (3.5,2.5)--(6,0) (6,0)--(6,5)
(6,5)--(3.5,2.5)  (6,5)--(2,5) (4,3)--(2,5);
 \node[left,magenta] at (3.5,2.5) {$O_k$};
 \filldraw[magenta](3.5,2.5) circle (2.5pt);
  \node[below,right,below,magenta] at (4,3) {$O_{k+1}$};
   \filldraw[magenta](4,3) circle (2.5pt);
 \node[below,blue] at (6,2.05) {$\frac{q_{k}}{2}$};
 \node[below,blue] at (6,3.75) {$\frac{q_{k}}{2}$};
\node[right,blue] at (6,0) {$Q_{k-1}$};
 \filldraw[blue](6,0) circle (2.5pt);
\node[right,blue] at (6,5) {$Q_k$};
 \filldraw[blue](6,5) circle (2.5pt);
\node[above,cyan] at (1,5) {$S_{k+1}$};
 \filldraw[yellow](1,5) circle (2.5pt);
\node[above,blue] at (2,5) {$Q_{k+1}$};
 \filldraw[blue](2,5) circle (2.5pt);
\draw[thick, dashed] (6,5)--(6,6);
\draw[thick, dashed] (2,5)--(1,5);
\draw[thick, dashed] (2,5)--(1,5);
\draw[thick, dashed] (3.5,2.5)--(6,2.5);
\coordinate (a) at (6,0); \coordinate (b) at (3.5,2.5); \coordinate (c) at (6,5); \draw pic[draw,angle radius=0.3cm,"$\theta$",left,below] {angle=a--b--c};
\coordinate (a) at (6,5); \coordinate (b) at (4,3); \coordinate (c) at (2,5); \draw pic[draw,angle radius=0.3cm,"$\theta$",above] {angle=a--b--c};
\coordinate (a) at (6,6); \coordinate (b) at (6,5); \coordinate (c) at (2,5); \draw pic[draw,angle radius=0.3cm,"$\theta$",above,left] {angle=a--b--c};
\node[black, below left] at (1,2.8) {$\mathfrak{C}_{k+1}$};
\node[black, below left] at (-0.5,2.8) {$\mathfrak{C}_k$};
\end{tikzpicture}
\caption{$\mathfrak{C}_{k}$ lies entirely outside $\mathfrak{C}_{k+1}$ and it only touches it at $Q_k$.
}\label{fig11}
\end{figure}


\section*{Competing interests}

The authors have no relevant financial or non-financial interests  to disclose.
\section*{Data Availability Statement}
Data sharing is not applicable to this article as no data sets were generated or analysed during the current study.


\bibliographystyle{abbrv}
{\bibliography{shymaaa_refs_01}}

\end{document}